\documentclass[reqno,11pt]{amsart}
\usepackage{latexsym}
\usepackage{amsfonts}
\usepackage{amssymb}
\usepackage{mathrsfs}
\usepackage[all]{xy}
\usepackage{bbm}
\usepackage{color}
\usepackage[mathscr]{euscript}

\newtheorem{theorem}{Theorem}[section]
\newtheorem{corollary}[theorem]{Corollary}
\newtheorem{lemma}[theorem]{Lemma}
\newtheorem{conj}[theorem]{Conjecture}

\theoremstyle{definition}

\newtheorem{remark}[theorem]{Remark}
\numberwithin{equation}{section}

\newcommand{\RR}{\mathbb{R}}
\newcommand{\NN}{\mathbb{N}}

\newcommand{\ZZ}{\mathbb{Z}}

\newcommand{\cR}{\Pi}

\newcommand{\cV}{\mathcal{V}}

\newcommand{\cN}{\mathcal{N}}
\newcommand{\cW}{\mathcal{W}}

\newcommand{\sB}{\mathscr{B}}
\newcommand{\Bad}{\mathbf{Bad}}
\newcommand{\bx}{\mathbf{x}}

\newcommand{\bv}{\mathbf{v}}

\newcommand{\bw}{\mathbf{w}}
\newcommand{\bP}{\mathbf{P}}

\newcommand{\diag}{\mathrm{diag}}
\newcommand{\dist}{\mathrm{dist}}

\newcommand{\SL}{\mathrm{SL}}

\newcommand{\SO}{\mathrm{SO}}

\newcommand{\Ad}{\mathrm{Ad}}
\newcommand{\Lie}{\mathrm{Lie}}
\newcommand{\fg}{\mathfrak{g}}
\newcommand{\fa}{\mathfrak{a}}
\newcommand{\fk}{\mathfrak{k}}
\newcommand{\fp}{\mathfrak{p}}
\newcommand{\fr}{\mathfrak{r}}
\newcommand{\fl}{\mathfrak{l}}

\newcommand {\ignore}[1]  {}

\newif\ifdraft\drafttrue

\draftfalse

\newcommand\hs{homogeneous space}

\newcommand{\rr}{\lambda}
\renewcommand{\emptyset}{\varnothing}
\renewcommand{\setminus}{\smallsetminus}

\begin{document}
	
	\title{Bounded Geodesics on Locally Symmetric Spaces}
	\author{Lifan Guan}
\address{Institute for Theoretical Sciences, School of Science, Westlake University,
600 Dunyu Road, Sandun town, Xihu district, Hangzhou 310030, Zhejiang Province, China.
}
\email{guanlifan@westlake.edu.cn}
\date{}
\author{Chengyang Wu}
\address{School of Mathematical Sciences, Peking University, Beijing, 100871, China
}
\email{chengyangwu@stu.pku.edu.cn}

	\begin{abstract}
    Let $\Gamma$ be a torsion-free subgroup of $\SL_3(\RR)$ commensurable with $\SL_3(\ZZ)$, and $Y=\SO_3(\RR)\backslash \SL_3(\RR)/\Gamma$ be endowed with the natural locally symmetric space structure. We prove that for any point $y\in Y$, the set of directions in which the geodesic ray starting from $y$ is bounded in $Y$, is hyperplane absolute winning.
	\end{abstract}

    \maketitle
	
	\section{Introduction}\label{S:intro}
	
	\subsection{Bounded geodesics on complete Riemannian manifolds}
	
	 Let $M$ be a complete Riemannian manifold, and let $S(M)$ denote its unit tangent bundle, whose fiber $S_p(M)$ over a point $p\in M$ is the unit sphere in $T_p(M)$ centered at the origin. For any $p\in M$ and any $\xi\in S_p(M)$, let $\gamma(p,\xi)$ denote the geodesic ray through the base point $p$ in the direction $\xi$. Of particular interest is the following subset of $S(M)$:
	 $$
	 E^+(\infty):=\{(p,\xi)\in S(M): \gamma(p,\xi) \text{ is bounded in } M\}
	 $$
	and its section over a point $p\in M$:
	$$
	E^+(p,\infty):=E^+(\infty)\cap S_p(M)=\{\xi\in S_p(M): \gamma(p,\xi) \text{ is bounded in } M\}.
	$$
	It is trivial to see that $E^+(\infty)=S(M)$ when $M$ itself is compact. But for a general noncompact complete Riemannian manifold $M$, it is even not clear whether $E^+(\infty)$ is nonempty. One can find in \cite{Sch00} an affirmative answer for a complete Riemannian manifold with $\dim(M)\geq 3$, curvature bounded by $[-b^2,-1]$ and finite volume.  
	
	In fact, one can say more for the special case when $M$ is a locally symmetric space of noncompact type with finite Riemannian volume. Here and hereafter we say that a subset $E$ in a metric space $S$ is \textit{thick} if $\dim(E\cap U)=\dim(U)$ for any open subset $U$ of $S$, where $\dim$ stands for the Hausdorff dimension on $S$. Dani \cite{Da1} proved in 1980s that $E^+(\infty)$ is thick in $S(M)$ for $M$ as above with a constant negative curvature (in particular of rank $1$). A decade later, Aravinda and Leuzinger \cite{Ara1,Ara2} proved that each $E^+(p,\infty)$ has winning property (see $\S2$) when restricted to a nonconstant $C^1$ curve, and hence is thick in $S_p(M)$ for $M$ as above of rank $1$. Similar results were also proved for the Teichmuller geodesic flow on the moduli space of unit-area holomorphic quadratic differentials by Kleinbock and Weiss \cite{KW04} and by Chaika, Cheung and Masur \cite{CCM11}.

    In this paper we first generalize above results as follows:

    \begin{theorem}\label{T:Gen1}
        Let $M$ be a locally symmetric space of non-compact type with finite Riemannian volume. Then for any $p\in M$, the set $E^+(p,\infty)$ is thick in $S_p(M)$.
    \end{theorem}

    We remark here that Theorem \ref{T:Gen1} easily follows from Corollary 5.5 in \cite{KM} and Marstrand slicing theorem. See Section \ref{S:pre} for details of the proof. 
    
    The main point of this paper is to consider a stronger property of sets than thickness. In 1960s, Schmidt \cite{Sc1} introduced a certain game played on any metric space, where a winning set for this game is proved to be thick. Later Schmidt's winning property was upgraded to an even stronger hyperplane absolute winning property (abbreviated as HAW) by Broderick, Fishman, Kleinbock, Reich and Weiss \cite{BFKRW} on any Euclidean space, and by Kleinbock and Weiss \cite{KW3} on any $C^1$ manifold. It should be stressed here that for any winning property mentioned above, it has the advantage that any countable intersection of winning sets is still winning. See Section \ref{S:2} for all definitions and discussions, and \cite{An1,An2,ABV,AGGL,AGK,AGK2,BPV,Be,BNY1,BNY2,BBFKW,BFK,BFS,KW2,KW1,NS} for other recent results involving winning properties of exceptional sets in Diophantine approximations and dynamical systems.

    In this terminology we raise a stronger conjecture than Theorem \ref{T:Gen1}:

    \begin{conj}\label{C:Gen1}
        Let $M$ be a locally symmetric space of non-compact type with finite Riemannian volume. Then for any $p\in M$, the set $E^+(p,\infty)$ is HAW on $S_p(M)$.
    \end{conj}
    
	Now we briefly review some basic facts concerning locally symmetric spaces. Let $M$ be a locally symmetric space of non-compact type, and let $\widetilde{M}$ denote its universal cover. The isometry group $G$ of $\widetilde{M}$ has finitely many connected components, whose identity component is a connected semisimple Lie group without compact factors and with trivial center. Fix a point $p_0\in M$, and let $\widetilde{p_0}\in\widetilde{M}$ be a preimage of $p_0$. The stabilizer $K:=\mathrm{Stab}_G(\widetilde{p_0})$ is a maximal compact subgroup of $G$. We identify the globally symmetric space $\widetilde{M}$ with $K\backslash G$, and view the fundamental group $\Gamma:=\pi_1(M)$ as a torsion-free discrete subgroup of $G$ via deck transformations. Then $M$ can be identified with $K\backslash G/\Gamma$. Conversely, given any semisimple Lie group $G$ without compact factors and with finite center, any maximal compact subgroup $K\subset G$ and any torsion-free discrete subgroup $\Gamma\subset G$, there exists a unique locally symmetric space structure on $K\backslash G/\Gamma$. Moreover, $M$ has finite Riemannnian volume if and only if $\Gamma$ is a lattice in $G$, and the rank of $M$ is the rank of $G$.
	

	In this paper we verify Conjecture \ref{C:Gen1} for a special kind of locally symmetric space of non-compact type with finite Riemannian volume and of rank $2$.

	\begin{theorem}\label{T:main1}
    Let $G=\SL_{3}(\RR),K=\SO_3(\RR)$, $\Gamma\leq G$ be a torsion-free subgroup commensurable with $\SL_3(\ZZ)$, and $Y=K\backslash G/\Gamma$ be endowed with the above natural locally symmetric space structure. Then Conjecture \ref{C:Gen1} holds for $Y$, that is, for any $y\in Y$, the set $E^+(y,\infty)$ is HAW on $S_y(Y)$.
	\end{theorem}

	
	
	\subsection{Relations between homogeneous dynamics and geometry}
	\label{S: homo and geo}
	Theorem \ref{T:main1} can also be reinterpreted from the view of homogeneous dynamics. Let $G$ be a semisimple Lie group without compact factors and with finite center, $\Gamma\subseteq G$ be a nonuniform torsion-free lattice, and $X=G/\Gamma$ be the corresponding homogeneous space. Then for any maximal compact subgroup $K$ of $G$, there exists a proper projection $$
	X=G/\Gamma\to Y=K\backslash G/\Gamma,
	$$
	which realizes a locally symmetric space as the quotient of a homogeneous space with compact fiber. 
	
	Now write $\fg=\Lie(G)$. There exists a Cartan involution $\theta$ with respect to the Killing form of $\fg$, which induces a Cartan decomposition $\fg=\fk\oplus\fp$. Here $\fk$ and $\fp$ are the $1$ and $-1$ eigenspaces of $\theta$ respectively. The involution $\theta$ also extends to an involutive automorphism $\widetilde{\theta}$ of $G$. Let $K=G_{\widetilde{\theta}}$ denote the subgroup of all $\widetilde{\theta}$-fixed elements in $G$. Then $K$ is a maximal compact subgroup of $G$, and there is a diffeomorphism $$
	K\times \fp\to G,\quad (k,X)\mapsto k\cdot\exp(X).
	$$
	It follows that the differential at $1_G$ of the projection $G\to K\backslash G$ gives a linear isomorphism $$
	\fp\to T_{[K 1_G]}(K\backslash G).
	$$
	By homogeneity the tangent space of each point $[Kg]\in K\backslash G$ is isomorphic to $\fp$. Moreover, the Riemannian metric on it can be seen as the $\Ad(K)$-invariant inner product on $\fp$. Then the unit tangent space equals $\fp^1:=\{\bv\in\fp: \|\bv\|=1\}$.
	
	For each base point $[Kg]\in K\backslash G$ and each direction $\xi=(\mathrm{d}R_g)(\bv)\in T_{[Kg]}(K\backslash G)$ with $\bv\in\fp^1$, the associated geodesic ray on $K\backslash G$ is given by $$
	\widetilde{\gamma}([Kg],\xi)=\{K\exp(t\bv)g: t\geq 0\}.
	$$
	Under the canonical covering map $K\backslash G\to K\backslash G/\Gamma$, the tangent space of each point $y=[Kg\Gamma]\in Y$ can be identified with that of $[Kg]\in K\backslash G$, and each geodesic ray in $Y$ through the base point $y$ in the direction $\xi$ can be written as $$
	\gamma(y,\xi)=\{K\exp(t\bv)g\Gamma: t\geq 0\}.
	$$
	It follows that 
	\begin{equation*}
		\begin{aligned}
			E^+(y,\infty)&=\{\xi\in S_y(Y): \gamma(y,\xi) \text{ is bounded in }Y\}\\
			&\cong\{\bv\in\fp^1: \{K\exp(t\bv)g\Gamma: t\geq 0\} \text{ is bounded in }Y\}\\
			&=\{\bv\in\fp^1: \{\exp(t\bv)g\Gamma: t\geq 0\} \text{ is bounded in }X\}.\\
		\end{aligned}
	\end{equation*}
	This reduces Theorem \ref{T:main1} into the framework of homogeneous dynamics. Since any countable intersection of HAW sets is still HAW and hence thick in the ambient space, we immediately get the following corollary of Theorem \ref{T:main1}:

    \begin{corollary}\label{C:main1}
    Let $G,\Gamma,K,\fp,Y$ be as in Theorem \ref{T:main1}. Then for any countably many points $\{y_n=[Kg_n\Gamma]\}_{n\in\NN}$ in $Y$, the set 
    $$\{\bv\in\fp^1:\{K\exp(t\bv)g_n\Gamma:t\geq 0\} \text{ is bounded in }Y \text{ for all }n\in\NN\}$$	
    is thick in $\fp^1$.
    \end{corollary}


	One can also consider all bounded diagonal rays through a point on the homogeneous space $X=G/\Gamma$. 
    Let $$
    \bP^+(\fg):=(\mathfrak{g}\setminus \{0\})/\sim,\quad \text{ where } \bv\sim\bv' \text{ if and only if } \bv\in \RR^+\bv'
    $$ denote the positive projective space of $\fg$. For each point $x\in X$, we write 
	$$ E^+(x,\infty):=\{[\bv]\in\bP^+(\fg): \{\exp(t\bv)x: t\geq 0\} \text{ is bounded in }X\}.
		$$
    The counterparts of Theorem \ref{T:Gen1}, Conjecture \ref{C:Gen1}, Theorem \ref{T:main1} and Corollary \ref{C:main1} are as follows:

    \begin{theorem}\label{T:Gen2}
        Let $G$ be a Lie group, $\Gamma\subseteq G$ be a lattice, and $X=G/\Gamma$ be the corresponding homogeneous space. Then for any $x\in X$, the set $E^+(x,\infty)$ is thick in $\bP^+(\fg)$.
    \end{theorem}

    \begin{conj}\label{C:Gen2}
        Let $G$ be a Lie group, $\Gamma\subseteq G$ be a lattice, and $X=G/\Gamma$ be the corresponding homogeneous space. Then for any $x\in X$, the set $E^+(x,\infty)$ is HAW on $\bP^+(\fg)$.
    \end{conj}
	
	\begin{theorem}\label{T:main2}
		Let $G=\SL_{3}(\RR),\Gamma=\SL_{3}(\ZZ)$, and $X= G/\Gamma$ be the corresponding homogeneous space. Then Conjecture \ref{C:Gen2} holds for $X$, that is, for any $x\in X$, the set $E^+(x,\infty)$ is HAW on $\bP^+(\fg)$.
	\end{theorem}
   
    \begin{corollary}\label{C:main2}
    Let $G,\Gamma,\fg,X$ be as in Theorem \ref{T:main2}. Then for any countably many points $\{x_n\}_{n\in\NN}$ in $X$, the set 
    $$\{[\bv]\in\bP^+(\fg): \{\exp(t\bv)x_n: t\geq 0\} \text{ is bounded in }X\text{ for all }n\in\NN\}$$	
    is thick in $\bP^+(\fg)$.
    \end{corollary}


    \begin{remark}
        Moreover, Conjecture \ref{C:Gen1} and \ref{C:Gen2} also hold for $G=\prod_{i=1}^s\SL_2(\RR)$ with a lattice $\Gamma$ and a maximal compact subgroup $K$ in $G$. This can be done by combining the arguments in this paper and \cite{AGGL}.
    \end{remark}


    

	\subsection{Organization of the paper and strategy of the proof}

	The paper is organized as follows. In Section \ref{S:2} we recall the definitions and properties of Schmidt's game and its variants--the hyperplane absolute game and the hyperplane potential game. In Section \ref{S:pre} we directly prove Theorems \ref{T:Gen1} and \ref{T:Gen2} by combining Corollary 5.5 in \cite{KM} and Marstrand slicing theorem. In Section \ref{S:main} we state our key result Theorem \ref{T:ru}, from which we deduce Theorems \ref{T:main1} and \ref{T:main2}. The proof of Theorem \ref{T:ru}, which forms the most technical part of this paper, is given in Sections 5-7.

    We highlight here that two of our main theorems, Theorems \ref{T:Gen1} and \ref{T:main1}, are formulated in a purely geometric way. They can be viewed as applications of the corresponding results in the field of homogeneous dynamics.

    Concerning the proof of Theorem \ref{T:ru}, the target set (\ref{defs}) can be regarded as a fiber bundle over an open interval of weight parameters, where each fiber is proved to be HAW on its ambient space (see \cite{AGK}). We manage to show that the winning strategy is ``locally constant'' with respect to the parameters, which is reflected in the choices of planes (\ref{E:plane}) in the proof of Lemma \ref{L:main2}. This requires a detailed analysis of how the strategy used in \cite{AGK} depends on the weight parameter.


    
	\section{Preliminaries on Schmidt games}\label{S:2}
	
	\subsection{Schmidt's $(\alpha,\beta)$-game}\label{alphabeta}
	
	We first recall Schmidt's $(\alpha,\beta)$-game introduced in \cite{Sc1}. It
	involves two parameters $\alpha,\beta\in(0,1)$  and is played by two players Alice and Bob on a complete metric space $(X,\dist)$ with a target set $S\subset X$. Bob starts the game by choosing a closed ball $B_0=B(x_0,\rho_0)$ in $X$ with center $x_0$ and radius $\rho_0$. After Bob chooses a closed ball $B_i = B(x_i,{\rho}_i) $, Alice chooses  $A_i = B(x_i', {\rho}_i')$ with $${\rho}'_i=\alpha {\rho}_i\text{ and }
	\dist(x'_i, x_i) \leq (1-\alpha){\rho}_i\,,$$
	and then Bob chooses  $B_{i+1} = B(x_{i+1},{\rho}_{i+1}) $ with $${\rho}_{i+1}=\beta {\rho}'_{i} \text{ and }  \dist(x_{i+1}, x'_i) \leq (1-\beta){\rho}'_i \,,$$
	etc. This implies that
	the balls are nested:
	$$
	B_0 \supset A_0 \supset B_1 \supset\cdots ;$$
	Alice wins the game if the unique point $\bigcap_{i=0}^\infty A_i=\bigcap_{i=0}^\infty B_i$ belongs to $S$, and Bob wins otherwise. The set $S$ is \textsl{$(\alpha,\beta)$-winning} if Alice has a winning strategy, is \textsl{$\alpha$-winning} if it is $(\alpha,\beta)$-winning for any $\beta\in(0,1)$, and is \textsl{winning} if it is $\alpha$-winning for some $\alpha$. Schmidt \cite{Sc1} proved that:
	
	\begin{itemize}
		\item[$\bullet$]  winning subsets of
		Riemannian manifolds
		are thick;
		\item[$\bullet$]  if $S$ is $\alpha$-winning and $\varphi:X\to X$ is bi-Lipschitz, then $\varphi(S)$ is   $\alpha'$-winning, where $\alpha'$ depends on $\alpha$ and the bi-Lipschitz constant of $\varphi$;
		\item[$\bullet$]  a countable intersection of $\alpha$-winning sets is again $\alpha$-winning.
	\end{itemize}
	
	As such, Schmidt's game has been a powerful tool for proving 
	thickness of intersections of certain countable families of sets, see e.g.\ \cite{An1, An2,  BBFKW, BFK, BFKRW, Da1, Da2, Da3}. However, for a fixed $\alpha$, the class of $\alpha$-winning subsets of a Riemannian manifold depends on the choice of the metric, and is not known to be preserved by diffeomorphisms.

	\subsection{Hyperplane absolute game on $\RR^d$}\label{haw}
	
	Inspired by ideas of McMullen \cite{Mc}, the hyperplane absolute game on the Euclidean space $\RR^d$ was introduced in \cite{BFKRW}.
	It has the advantage that the family of its winning sets is preserved by $C^1$ diffeomorphisms.
	Let $S \subset \RR^d$ be a target set and let
	$\beta \in \left(0, \frac13 \right)$.
	As before Bob begins by choosing a closed ball $B_0$ of radius ${\rho}_0$.
	For an affine hyperplane $L\subset\RR^d$ and $\rho>0$, we denote the ${\rho}$-neighborhood of $L$ by
	$$L^{({\rho})}:=\{\bx\in\RR^d:\mathrm{dist}(\bx,L)<{\rho}\}.$$
	Now, after Bob chooses a closed ball $B_i$ of radius ${\rho}_i$, Alice chooses a hyperplane neighborhood $L_i^{({\rho}'_i)}$ with ${\rho}'_i\le\beta {\rho}_i$,
	and then Bob chooses a closed ball $B_{i+1}\subset B_i\setminus L_i^{({\rho}'_i)}$ of radius ${\rho}_{i+1}\ge\beta {\rho}_i$. Alice wins the game if and only if
	$$\bigcap_{i=0}^\infty B_i\cap S\ne\emptyset.$$ The set $S$ is \textsl{$\beta$-hyperplane absolute
		winning} (\textsl{$\beta$-HAW} for short)
	if Alice has a winning strategy, and is \textsl{hyperplane absolute winning} (\textsl{HAW} for short)
	if it is $\beta$-HAW for any $\beta\in(0,\frac{1}{3})$.

	\begin{lemma}[\cite{BFKRW}]\label{L:HAW-Rd}
		\begin{itemize}
			\item[(i)] HAW subsets   are winning, and hence thick.
			\item[(ii)] A countable intersection of HAW subsets   is again HAW.
			\item[(iii)]   The image of an HAW set under a $C^1$ diffeomorphism $\RR^d \to \RR^d$ is HAW.
		\end{itemize}
	\end{lemma}
	
	\subsection{HAW subsets of a manifold}\label{hawmfld}
	
	The notion of HAW sets has been extended to subsets of $C^1$ manifolds in \cite{KW3}.
	This is done in two steps. First one defines the absolute hyperplane game on an open subset $W
	\subset \RR^d$. It is defined just as the absolute hyperplane game on
	$\RR^d$, except for requiring that Bob's first move $B_0$ be
	contained in $W$. If Alice has a winning strategy, we say that $S$ is
	{\sl HAW on $W$}.
	Now let $M$ be a $d$-dimensional $C^1$ manifold, and let $\{(U_\alpha, \varphi_\alpha)\}$ be a $C^1$ atlas, that is, $\{U_\alpha\}$ is an open cover of $M$, and each $\varphi_\alpha$ is a $C^1$ diffeomorphism from $U_\alpha$ onto the open subset $\varphi_\alpha(U_\alpha)$ of $\RR^d$. A subset $S\subset M$ is said to be \textsl{HAW on $M$} if for each $\alpha$, $\varphi_\alpha(S\cap U_\alpha)$ is HAW on $\varphi_\alpha(U_\alpha)$. Note that Lemma \ref{L:HAW-Rd} (iii) implies that the definition is independent of the choice of the atlas (see \cite{KW3} for details).
	
	\begin{lemma}[\cite{KW3}]\label{L:HAW-mnfd}
		\begin{itemize}
			\item[(i)]  HAW subsets of a $C^1$ manifold are thick.
			\item[(ii)]  A countable intersection of HAW subsets of a $C^1$ manifold is again HAW.
			\item[(iii)] Let $\varphi:M\to N$ be a diffeomorphism between $C^1$ manifolds, and let $S\subset M$ be an HAW subset of $M$. Then $\varphi(S)$ is an HAW subset of $N$.
			\item[(iv)] Let $M$ be a $C^1$ manifold with an open cover $\{U_\alpha\}$. Then a subset $S\subset M$ is HAW on $M$ if and only if $S\cap U_\alpha$ is HAW on $U_\alpha$ for each $\alpha$.
			\item[(v)] Let $M, N$ be $C^1$ manifolds, and let $S\subset M$ be an HAW subset of $M$. Then $S\times N$ is an HAW subset of $M\times N$.
		\end{itemize}
	\end{lemma}
	
	\begin{proof}
		(i)--(iii) appeared as \cite[Proposition 3.5]{KW3} and are clear from Lemma \ref{L:HAW-Rd}. (iv) is clear from the definition. (v) is proved in \cite[Proof of Theorem 3.6(a)]{KW3}
	\end{proof}
	
	We shall also need the following lemma,
	\begin{lemma}\label{L:HAW-open}
		Let $M$ be a $C^1$ manifold of dimension $d$, $M'\subset M$ be a finite union of $C^1$ submanifolds of strict lower dimensions. Then $S\subset M\setminus M'$ is HAW on $M\setminus M'$ if and only if it is HAW on $M$.
	\end{lemma}
	
	\begin{proof}
		The if direction is a direct consequence of the definition. For the only if direction, by using an induction argument, we may assume that $M'$ itself is an $C^1$ submanifold of strict lower dimension. We choose and fix a $C^1$-altas $\{(U_\alpha,\phi_\alpha)\}$ of $M$ such that, for any $\alpha$, there exists a hyperplane $L_{\alpha}\subset \RR^d$ satisfying
		$$\phi_\alpha(U_\alpha\cap M')\subset L_{\alpha}\subset \RR^d.$$
		Since the HAW property is independent of choices of the atlas, to prove $S$ is HAW on $M$, it suffices to prove $\phi_{\alpha}(S\cap U_\alpha)$ is HAW on $\phi_{\alpha}(U_{\alpha})$ for all $U_\alpha$ in this specific $C^1$-altas. As $\left\{\left(U_\alpha\cap\left(M\setminus M'\right),\phi_\alpha\right)\right\}$ is a $C^1$-altas of $M\setminus M'$, $\phi_{\alpha}\left(S\cap U_\alpha\cap\left(M\setminus M'\right)\right)$ is HAW on $\phi_{\alpha}\left(U_{\alpha}\cap \left(M\setminus M'\right)\right)$. In view of these properties, the winning strategy for Alice is clear: for any $\beta$-hyperplane absolute game played on $\phi_{\alpha}(U_{\alpha})$, after Bob choosing his $B_0$, Alice chooses neighborhood $L_{\alpha}^{\beta\rho(B_0)}$ as her first move. Then the rest game is played on $\phi_{\alpha}\left(U_{\alpha}\cap \left(M\setminus M'\right)\right)$, so Alice can play according to her winning strategy on this set. This completes the proof.
	\end{proof}

	\subsection{Hyperplane potential game}
	
	Finally, we recall the hyperplane potential game introduced in \cite{FSU}. Being played  on $\RR^d$, it has the same winning sets as the hyperplane absolute game. This allows one to prove the HAW property of a set $S\subset\RR^d$ by showing that it is winning for the hyperplane potential game (see \cite{NS}).
	
	Let $S\subset\RR^d$ be a target set, and let $\beta\in(0,1)$, $\gamma>0$. The \textsl{$(\beta,\gamma)$-hyperplane potential game} is defined as follows: Bob begins by choosing a closed ball $B_0\subset\RR^d$. After Bob chooses a closed ball $B_i$ of radius ${\rho}_i$, Alice chooses a countable family of hyperplane neighborhoods $\{L_{i,k}^{({\rho}_{i,k})} : k\in\NN\}$ such that
	$$\sum_{k=1}^\infty {\rho}_{i,k}^\gamma\le(\beta {\rho}_i)^\gamma,$$
	and then Bob chooses a closed ball $B_{i+1}\subset B_i$ of radius $\rho_{i+1}\ge\beta \rho_i$. Alice wins the game if and only if
	$$\bigcap_{i=0}^\infty B_i\cap\Big(S\cup\bigcup_{i=0}^\infty\bigcup_{k=1}^\infty L_{i,k}^{(\rho_{i,k})}\Big)\ne\emptyset.$$
	The set $S$ is \textsl{$(\beta,\gamma)$-hyperplane potential winning} (\textsl{$(\beta,\gamma)$-HPW} for short) if Alice has a winning
	strategy, and is \textsl{hyperplane potential winning} (\textsl{HPW} for short) if it is $(\beta,\gamma)$-HPW for any $\beta\in(0,1)$ and $\gamma>0$. The following lemma is a special case of \cite[Theorem C.8]{FSU}.
	
	\begin{lemma}\label{L:HPW}
		A subset $S\subset\RR^d$ is HPW if and only if it is HAW.
	\end{lemma}

	\section{Proof of Theorems \ref{T:Gen1} and \ref{T:Gen2}}\label{S:pre}

     In this section we prove Theorems \ref{T:Gen1} and \ref{T:Gen2} by reducing them to Corollary 5.5 in \cite{KM}. For the sake of convenience we state one version of the aforementioned result here.

    \begin{theorem}[Corollary 5.5 in \cite{KM}]\label{T:KM1.5} Let $G$ be a Lie group, $\Gamma$ be a lattice in $G$, $F$ be a one-parameter $\Ad$-diagonalizable subgroup of $G$, $H=H(F^+)$ be the unstable horospherical subgroup with respect to $F^+$, and $X=G/\Gamma$. Then for any $x\in X$, the set $$
    E_H^+(x,\infty):=\{h\in H:F^+hx\text{ is bounded in }X\}
    $$
    is thick in $H$.        
    \end{theorem}

    \begin{proof}[Proof of Theorem \ref{T:Gen1}] We first recall some notations in Section \ref{S:intro}. Let $M$ be a locally symmetric space of non-compact type with finite Riemannian volume. Then $M$ can be identified with $Y=K\backslash G/\Gamma$, where $G$ is a semisimple Lie group without compact factors and with finite center, $K$ is a maximal compact subgroup of $G$, and $\Gamma$ is a lattice in $G$. Let $X=G/\Gamma$, $\fg=\Lie(G),\fk=\Lie(K)$, and $\fp\subset \fg$ be the orthogonal complement of $\fk$. Then it suffices to show that: for any point $y=[Kg\Gamma]\in Y$, the set 
    $$
    \{\bv\in\fp^1: \{\exp(t\bv)g\Gamma: t\geq 0\} \text{ is bounded in }X\}
    $$
    is thick in $\fp^1$.

    Let us choose a maximal abelian subspace $\fa\subset\fp$ and an open Weyl chamber $\fa_+\subset \fa$. Let $L=Z_K(\fa)$ and $\fa^1_+=\fa_+\cap \fp^1$. Then the map $$
    \theta: \fa^1_+\times L\backslash K\to \fp^1, \quad (\bv, Lk)\mapsto (\Ad k^{-1})\bv
    $$
    is a diffeomorphism onto an open dense subset of $\fp^1$. Then it suffices to prove that the set $$\{(\bv,Lk)\in \fa^1_+\times L\backslash K:\{\exp(t\bv)kg\Gamma: t\geq 0\} \text{ is bounded in }X\}$$ is thick in $\fa^1_+\times L\backslash K$. Moreover, by Marstrand's slicing and projection theorems, we only need to prove that: for any $\bv\in\fa^1_+$, the set $$
    E_K^+([g\Gamma],\infty):=\{k\in K: \{\exp(t\bv)kg\Gamma: t\geq 0\} \text{ is bounded in }X\}
    $$ 
    is thick in $K$. 

    Let $F_{\bv}^+=\{\exp(t\bv): t\geq 0\}$ be the one-parameter subsemigroup of $G$, and $H=H(F_\bv^+)$ (resp. $H^0=H^0(F_\bv^+),H^-=H^-(F_\bv^+)$)  be the unstable (resp. neutral, stable) horospherical subgroup with respect to $F_\bv^+$. 
    Let $B=H^0H$ and $B^-=H^-H^0$, and consider the action of the subgroup $K$ on the real flag variety $G/B$ by left translation. By Iwasawa decomposition, this action is transitive. It follows that for any $g\in G$, the map $$\phi_g: K\rightarrow G/B,\quad k\mapsto kgB$$ is a projection map with compact fiber. Write $K'=K\cap B^-Bg^{-1}=K\cap B^-Hg^{-1}$. Then $K\setminus K'=\phi_g^{-1}((G\setminus B^{-1}B)/B)$  is a finite union of manifolds of strict lower dimensions. So it suffices to show that $E_K^+([g\Gamma],\infty)\cap K'$ is thick in $K'$. 
    
    For $k\in K'$, we set $\sigma_1(k)\in B^-$ and $\sigma_2(k)\in H$ to be the unique pair satisfying $kg=\sigma_1(k)\sigma_2(k)$. It is easy to check that the map $\sigma_2$ is a projection map. On the other hand,
		\begin{equation*}
			F_{\bv}^+ kg\Gamma \text{ is bounded } \Longleftrightarrow  F_{\bv}^+ \sigma_2(k)\Gamma \text{ is bounded}.
		\end{equation*}
		Consequently, 
        $$E_K^+([g\Gamma],\infty)\cap K'=\sigma_2^{-1}(\{h\in H:F_{\bv}^+h\Gamma \text{ is bounded in }X\})=\sigma_2^{-1}(E_H^+([1_G\Gamma],\infty)).$$ 
        In view of Theorem \ref{T:KM1.5}, we see that $E_H^+([1_G\Gamma],\infty)$ is thick in $H$, and hence by Marstrand's slicing theorem $E_K^+([g\Gamma],\infty)\cap K'$ is thick in $K'$.
        This completes the proof.
    \end{proof}

    \begin{proof}[Proof of Theorem \ref{T:Gen2}] We begin with a simple observation, that is, it suffices to prove Theorem \ref{T:main2} for $x=[1_G\Gamma]\in X$. Indeed, for any $x'=[g\Gamma]\in X$ and one-parameter subgroup $F\subset G$, $F^+x'$ is bounded if and only if $g^{-1}F^+gx$ is bounded, which implies
		$$E^+(x', \infty)=\Ad(g)\left(E^+(x,\infty)\right).$$
		Thus it suffices to prove Theorem \ref{T:Gen2} for $x=[1_G\Gamma]\in X$.

    Let $p:\fg\setminus \{0\}\to \bP^+(\fg)$ denote the projection map. By Marstrand's projection and slicing theorems, to prove Theorem \ref{T:Gen2} is equivalent to show that the set $$
    \{\bv\in\fg: F_{\bv}^+\Gamma\text{ is bounded in }X\}
    $$
    is thick in $\fg$. We claim that we may reduce to the case when $G$ is a connected semisimple Lie group with trivial center and without compact factors for reasons below:

    First, let $G^0$ denote the identity component of $G$. It is clear that $\Lie(G^0)=\fg$ and that $X^0=G^0/(G^0\cap \Gamma)$ is a connected component of $X$ containing $[1_G\Gamma]$. In particular, $$
    F_{\bv}^+\Gamma\text{ is bounded in }X\Longleftrightarrow 
    F_{\bv}^+\Gamma\text{ is bounded in }X^0.
    $$
    So we may assume that $G=G^0$ is connected without loss of generality.
    
    Second, let $\fg=\fr\rtimes \fl$ be the Levi decomposition, where $\fr$ is the solvable radical of $\fg$ and $\fl$ is a Levi subalgebra of $\fg$. Let $\fl=\fl_c\oplus \fl_{nc}$ be the decomposition into simple ideals where $\fl_c$ is the direct sum of all compact simple ones. The projection map $\mathrm{d}q:\fg\to \fl_{nc}$ is a Lie algebra homomorphism, which induces a Lie group homomorphism $q:G\to L_{nc}$ and a projection map $X=G/\Gamma\to X_{nc}:=L_{nc}/q(\Gamma)$. Note that $q(\Gamma)$ is a lattice in $L_{nc}$, and the fiber of $X\to X_{nc}$ is a compact homogeneous space $RL_c/(RL_c\cap \Gamma)$ (see \cite{Da0}, Lemma 9.1). Then for $\bv\in\fg$, $$
    F_{\bv}^+\Gamma\text{ is bounded in }X\Longleftrightarrow 
    q(F_{\bv})^+q(\Gamma)=F_{\mathrm{d}q(\bv)}^+q(\Gamma) \text{ is bounded in }X_{nc}.
    $$
    It follows that 
    $$
    \{\bv\in\fg: F_{\bv}^+\Gamma\text{ is bounded in }X\}
    =(\mathrm{d}q)^{-1}\{\bw\in\fl_{nc}: F_{\bw}^+q(\Gamma)\text{ is bounded in }X_{nc}\}.
    $$
    By Marstrand's slicing theorem, it suffices to show that $$
    \{\bw\in\fl_{nc}: F_{\bw}^+q(\Gamma)\text{ is bounded in }X_{nc}\}
    $$
    is thick in $\fl_{nc}$. So we may assume that $G$ is a semisimple Lie group without compact factors.
    
    Finally, since $X=G/\Gamma$ is a finite cover of $G/(Z(G)\cdot \Gamma)$, we may replace $G$ by $G/Z(G)$ and $\Gamma$ by $(\Gamma\cdot Z(G))/Z(G)$ without loss of generality. 
    These verify the claim.
    
    Now let $G$ be a connected semisimple Lie group with trivial center and without compact factors. Let $G=K\cdot \exp(\fp)$ be the global Cartan decomposition, $\fa\subset\fp$ be a maximal abelian subspace, and $L=Z_K(\fa)$. We choose an open Weyl chamber $\fa_+\subset\fa$, and for any $\bv\in \fa_+$, let $H=H(F_\bv^+)$ (resp. $H^0=H^0(F_\bv^+),H^-=H^-(F_\bv^+)$)  be the unstable (resp. neutral, stable) horospherical subgroup with respect to $F_\bv^+$. Note that the subgroups $H,H^0,H^-$ only depend on the choice of the open Weyl chamber $\fa_+$. 

    Consider the map $$\theta_0: \fa_+\times Z_G(\fa_+)\backslash G \rightarrow \fg, \quad (\bv,Z_G(\fa_+)g)\mapsto \Ad(g^{-1})\bv.$$
    It is easy to check that $\theta$ is a diffeomorphism onto $\fg$. Using Bruhat decomposition, there is a diffeomorphism onto its image: $$
    H\times L\times H^-\to Z_G(\fa_+)\backslash G, \quad (h_1,h_2,h_3)\mapsto Z_G(\fa_+)h_3h_2h_1.
    $$
    It follows that the composition map
		$$\theta: \fa_+\times H\times L\times H^- \rightarrow \fg, \quad (\bv, h_1,h_2,h_3)\mapsto\Ad(h_1^{-1}h_2^{-1}h_3^{-1})\bv$$
     is also a diffeomorphism onto its image. Moreover, we see that $p$ is a projection map when restricted to $\mathrm{Im}(\theta)$, and that $\fg\setminus \mathrm{Im}\theta$ is a finite union of submanifolds of strict lower dimensions. Thus it suffices to prove that $$
    \{(\bv,h_1,h_2,h_3)\in\fa_+\times H\times L\times H^-: F_{\bv}^+h_3h_2h_1\Gamma\text{ is bounded in }X\}
    $$
    is thick on $\fa_+\times H\times L\times H^-$. But we have for any $\bv\in\fa_+$ $$
    F_{\bv}^+h_3h_2h_1\Gamma\text{ is bounded in }X\Longleftrightarrow F_{\bv}^+h_1\Gamma\text{ is bounded in }X.
    $$
    So the conclusion follows from Marstrand's slicing theorem and Theorem \ref{T:KM1.5}. This completes the proof.
    \end{proof}

	\section{Proof of Theorems \ref{T:main1} and \ref{T:main2}}\label{S:main}


In this section we prove Theorems \ref{T:main1} and \ref{T:main2} by assuming the truth of Theorem \ref{T:ru} below. The proof of Theorem \ref{T:ru} will be postponed to the next section.
 
	For any $\rr \in \cR:=(1/2,1)$, set
	\begin{equation*}\label{sldgt}
		\bv_\rr=\diag(\lambda,1-\lambda,-1)\in \fg,  F_{\rr}^+ = \{\exp(t\bv_{\rr}): t \ge 0\}\text{ and } F_{\rr}^- = \{\exp(-t\bv_{\rr}): t \ge 0\}.
	\end{equation*}
	Let $ U\subset G$ (resp. $U^- \subset G$) be the unipotent subgroup of upper triangular (resp. lower triangular) matrices and $A\subset G$ be the subgroup of diagonal matrices with positive entries.
	Then $U$ (resp. $U^-$) is the expanding (resp. contracting) subgroup of $F_{\rr}^+$ for any $\rr\in \cR$. Define
	\begin{equation*}\label{def-r-u}
		E_U(F_{\rr}^+):=\{u\in U: F_{\rr}^+u\Gamma \text{ is bounded in }X\}.
	\end{equation*}
	
	Both Theorem \ref{T:main1} and Theorem \ref{T:main2} will be deduced from the following theorem.
	\begin{theorem}\label{T:ru}
		The set
		\begin{equation}\label{defs}
			S:=\{(\rr,u)\in \cR\times U: u \in E_U(F_{\rr}^+)\}
		\end{equation}
		is HAW on $\cR\times U$.
	\end{theorem}
	\begin{remark}\label{r:slice}
		In \cite{AGK}, it is proved that $E_U(F_{\rr}^+)$ is HAW on $U$ for any $\rr\in \cR$. But Theorem \ref{T:ru} does not follow from this result. It is generally believed that for a general Borel subset in a product space, winning properties of all its slices do not necessarily imply the winning property of itself.
	\end{remark}

\begin{proof}[Proof of Theorem \ref{T:main1} modulo Theorem \ref{T:ru}.]
		We first recall basic facts about geodesic flows on $Y$ from \cite{Mau}.
		Let $\fk=\Lie(K)$ and $\fp\subset \fg$ be its orthogonal complement. Then for any $y=[Kg\Gamma]\in Y$, $S_y(Y)$ can be identified with $\fp^1=\{\bv\in \fp: \|\bv\|=1\}$. 
        Moreover, for any $y=[Kg\Gamma]\in Y$ and $\bv\in \fp^1$, the associated geodesic flow is given by
		\begin{equation*}\label{E:geodesic}\gamma(y,\bv)=\{K\exp(t\bv)g\Gamma: t\geq 0\}.\end{equation*}
		Since $K$ is compact, 
        we have 
        $$E^+(y,\infty)=\{\bv\in \fp^1: F_{\bv}^+g\Gamma \text{ is bounded in }X\}.$$

		The proof of Theorem \ref{T:main1} will be given into two steps. To begin, we consider the map
		\begin{equation*}
			\theta: \cR\times K\rightarrow \fp^1, \quad \theta(\rr, k)=\Ad(k^{-1})\bv_\rr',
			\text{ where } \bv_\rr'=\|\bv_\rr\|^{-1}\bv_\rr\in \fp^1.
		\end{equation*}
		It is easy to check that $\fp^1\setminus \big(\mathrm{Im}\theta\cup -\mathrm{Im}\theta\big)$ is a finite union of manifolds of strict lower dimension. Hence in view of Lemma \ref{L:HAW-open}, it suffices to prove that $E^+(y,\infty)\cap \pm \mathrm{Im}\theta$ is HAW on $\pm \mathrm{Im}\theta$. By applying $\overline{\tau}:X\rightarrow X$ induced by the outer automorphism $\tau(g)=(g^T)^{-1}$ defined on $G$, we get
		$$E^+(y,\infty)\cap -\mathrm{Im}\theta=\mathrm{d}\tau(E^+(y',\infty)\cap \mathrm{Im}\theta), \text{ where }y'=(g^T)^{-1}\Gamma.$$
		Thus it suffices to show $E^+(y,\infty)\cap \mathrm{Im}\theta$ is HAW on $\mathrm{Im}\theta$ for any $y\in Y$. On the other hand, it is also easy to check that $\theta$ is a cover map. Hence in view of Lemma \ref{L:HAW-mnfd}(iii), it suffices to show that
		$$E':=\theta^{-1}(E^+(y,\infty))=\{(\rr,k)\in \cR\times K: F_\rr^+ kg\Gamma \text{ is bounded in }X\}.$$
		is HAW on $\cR\times K$.

		Let $ B\subset G$ (resp. $B^- \subset G$) be the subgroup of upper triangular (resp. lower triangular) matrices, and consider the action of the subgroup $K$ on the real flag variety $G/B$ by left translations. By Iwasawa's decomposition, this action is transitive and has finite stabilizer. It follows that for any $g\in G$, the map $\phi_g: K\rightarrow G/B$ defined by $\phi_g(k)= kgB$ is a cover map. 
        Write $K'=K\cap B^-Bg^{-1}=K\cap B^-Ug^{-1}$. Then $K\setminus K'=\phi_g^{-1}((G\setminus B^{-1}B)/B)$  is a finite union of manifolds of strict lower dimensions. Using Lemma \ref{L:HAW-open} again, it suffices to show that $E'':=E'\cap (\cR\times K')$ is HAW on $\cR\times K'$. For $k\in K'$, we set $\sigma_1(k)\in B^-$ and $\sigma_2(k)\in U$ to be the unique pair satisfying $kg=\sigma_1(k)\sigma_2(k)$. It is easy to check that the map $\sigma_2$ is a cover map. On the other hand,
		\begin{equation*}
			F_\rr^+ kg\Gamma \text{ is bounded in }X \Longleftrightarrow  F_\rr^+ \sigma_2(k)\Gamma \text{ is bounded in }X.
		\end{equation*}
		Consequently, $E''=(\mathrm{Id}\times \sigma_2)^{-1}S$. Thus in view of Lemma \ref{L:HAW-mnfd}(iii) and Theorem \ref{T:ru}, we see that $E''$ is HAW on $\cR\times K'$.
        This completes the proof.
	\end{proof}

	\begin{proof}[Proof of Theorem \ref{T:main2} modulo Theorem \ref{T:ru}.]
		We begin with a simple observation, that is, it suffices to prove Theorem \ref{T:main2} for $x=[1_G\Gamma]\in X$. Indeed, for any $x'=[g\Gamma]\in X$ and one-parameter subgroup $F\subset G$, $Fx'$ is bounded if and only if $g^{-1}Fgx$ is bounded, which implies
		$$E^+(x', \infty)=\Ad(g)\left(E^+(x,\infty)\right).$$
		Hence in view of Lemma \ref{L:HAW-mnfd} (iii), it suffices to prove Theorem \ref{T:main2} for $x=[1_G\Gamma]\in X$.
		
		Consider the map
        $$\theta: \cR\times U\times U^- \rightarrow \fg, \quad (\rr, u_1, u_2)\mapsto\Ad(u_1^{-1}u_2^{-1})\bv_\rr$$
        It is straightforward to check that the map $\theta$ is a diffeomorphism onto its image. Let $p$ be the projection from $\fg\setminus\{0\}$ to $\bP^+(\fg)$. Then it can be checked that $p$ is a diffeomorphism onto its image when restricted on $\mathrm{Im}\theta\cup-\mathrm{Im\theta}$ and $\bP^+(\fg)\setminus (p(\mathrm{Im}\theta)\cup -p(\mathrm{Im}\theta))$ is a finite union of submanifolds of strict lower dimensions. In view of Lemma \ref{L:HAW-mnfd}(iv) and Lemma \ref{L:HAW-open}, it suffices to show that $E^+([1_G\Gamma],\infty)\cap \pm p(\mathrm{Im}\theta)$ is HAW on $\pm p(\mathrm{Im}\theta)$, respectively.      
        Set
		$$E^+(\infty)=\left\{\bv\in \fg: F_{\bv}^+\Gamma \text{ is bounded in  }  X \right\}.$$ Then $E^+([1_G\Gamma],\infty)=p(E^+(\infty)\setminus \{0\})$. We claim that Theorem \ref{T:main2} can be deduced from the following statement
		\begin{equation}\label{claim1}
			E^+(\infty)\cap \mathrm{Im}\theta \text{ is HAW on }  \mathrm{Im}\theta.
		\end{equation}
        
        Indeed, on one hand, it follows from Lemma \ref{L:HAW-mnfd}(iii) and (\ref{claim1}) that $$
        E^+([1_G\Gamma],\infty)\cap p(\mathrm{Im}\theta)=p(E^+(\infty)\cap\mathrm{Im}\theta) \text{ is HAW on } p(\mathrm{Im}\theta).
        $$
        On the other hand, by applying the map $\overline{\tau}:X\to X$ induced by the outer automorphism $\tau(g)=(g^T)^{-1}$ defined on $G$, we see that $$
        E^+([1_G\Gamma],\infty)\cap -p(\mathrm{Im}\theta)=
        (\mathrm{d}\tau\circ p)(E^+(\infty)\cap\mathrm{Im}\theta).
        $$
        It follows again from Lemma \ref{L:HAW-mnfd}(iii) and (\ref{claim1}) that $$
        E^+([1_G\Gamma],\infty)\cap -p(\mathrm{Im}\theta)\text{ is HAW on } (\mathrm{d}\tau\circ p)(\mathrm{Im}\theta)=-p(\mathrm{Im}\theta).
        $$
        These verify our claim.

        Finally, to prove (\ref{claim1}) it suffices to establish the HAW property of $\theta^{-1} \big(E^+(\infty)\big)$ on $\Pi\times U\times U^-$. But we have
		\begin{align*}
			(\rr, u_1,u_2)\in \theta^{-1} \big(E^+(\infty)\big)
			&\Longleftrightarrow F_\rr^{+}u_2u_1\Gamma \text{ is bounded }\\
			&\Longleftrightarrow (\rr, u_1)\in S,
		\end{align*}
		which means that $\theta^{-1} \big(E^+(\infty)\big)=S\times U^-$. So the conclusion follows from Lemma \ref{L:HAW-mnfd}(v) and Theorem \ref{T:ru}.
        This completes the proof.
	\end{proof}

		\section{Proof of Theorem \ref{T:ru}}

In this section we prove Theorem \ref{T:ru} modulo two key lemmas. 
        
		For $(x,y,z)\in\RR^3$, set
		\begin{equation*}\label{defu}
			u_{x,y,z}:=\begin{pmatrix}
				1&z&x \\&1&y\\&&1
			\end{pmatrix}\in U.
		\end{equation*}
		
		For technical reasons, we will prove Theorem \ref{T:ru} by applying Lemma \ref{L:HAW-mnfd} (iii) to the diffeomorphism $\Psi:\Pi\times \RR^3\to \cR\times U$ defined by
		$$\Psi\big(\rr,(x,y,z)\big)=\left(\lambda, u_{x,y,z}^{-1}\right).$$
		For simplicity, we write $\Sigma:=\Pi\times \RR^3$. It is directly checked that
		$$\Psi^{-1}(S)=\left\{\big(\lambda,(x,y,z)\big)\in \Sigma: F_\lambda^+u_{x,y,z}^{-1} \text{ is bounded}\right\}.$$
		So it suffices to prove $\Psi^{-1}(S)$ is HAW on $\Sigma$.
		
		For $\lambda\in\Pi$, $\epsilon>0$ and $\bv=(p,r,q)\in\ZZ^2\times\NN$, we set
		$$\Delta_\lambda(\bv,\epsilon):=\Big\{(x,y,z)\in\RR^3:\Big|x-\frac{p}{q}-z\Big(y-\frac{r}{q}\Big)\Big|<\frac{\epsilon}{q^{1+\lambda}}, \Big|y-\frac{r}{q}\Big|<\frac{\epsilon}{q^{2-\lambda}}\Big\},$$
		and
		\begin{equation}\label{E:S-epsilon}
			S_\lambda(\epsilon):=\RR^3\setminus\bigcup_{\bv\in\ZZ^2\times\NN}\Delta_\lambda(\bv,\epsilon).
		\end{equation}
		Note that if $\bv'=t\bv $ with $t>1$, then $\Delta_\lambda(\bv',\epsilon)\subset \Delta_\lambda(\bv,\epsilon)$. Thus,
		\begin{equation*}
			S_\lambda(\epsilon)=\RR^3\setminus\bigcup_{\bv\in\Xi}\Delta_\lambda(\bv,\epsilon), \text{ where } \Xi=\{(p,r,q)\in\ZZ^2\times\NN: \gcd(p,r,q)=1\}.
		\end{equation*}
		Moreover, we set
		\begin{equation*}
			\Delta(\bv,\epsilon)=\bigcup_{\lambda\in \Pi} \{\lambda\}\times \Delta_\lambda(\bv,\epsilon),\text{ and }  S(\epsilon):=\Sigma\setminus \bigcup_{\bv\in\Xi}\Delta(\bv,\epsilon)= \bigcup_{\lambda\in \Pi} \{\lambda\}\times S_\lambda(\epsilon).
		\end{equation*}
		
		According to \cite[Lemma 3.2]{AGK}, we know that $\{\lambda\}\times S_\lambda(\epsilon)\subset \Psi^{-1}(S)$ for any
		$\lambda\in \Pi$ and $\epsilon>0$. Hence $S(\epsilon)\subset \Psi^{-1}(S)$ for any $\epsilon>0$. Therefore, Theorem \ref{T:ru} can be deduced from the following lemma:
		\begin{lemma}\label{statement1}
			The set $S'=\bigcup_{\epsilon>0} S(\epsilon)$ is HAW on $\Sigma$.
		\end{lemma}
		
		The rest of this section is devoted to the proof of Lemma \ref{statement1}. For convenience, we endow $\Sigma$ with the product metric. For a closed ball $\Omega\subset \Sigma$, we denote its radius as $\rho_{\Omega}$ and its center as $(\lambda_\Omega, \omega_\Omega)$  with $\lambda_\Omega\in \Pi,\;\omega_\Omega=(x_\Omega, y_\Omega, z_\Omega)\in \RR^3$. Hence, we have
		$$\Omega=I_{\Omega}\times B_{\Omega}, \text{ where }I_{\Omega}=I(\lambda_\Omega,\rho_\Omega)\text{ and } B_{\Omega}= B(\omega_\Omega, \rho_\Omega).$$
		
		For a closed ball $\Omega\subset \Sigma$ and $\bv=(p,r,q)\in\Xi$, consider the set of integral vectors
		\begin{align*}
			\cW(\Omega,\bv):=\big\{(a,b,c)\in\ZZ^3: \ &  (a,b)\ne(0,0), ap+br+cq=0,\\
			& |a|\le e^{\sqrt{\rho_\Omega}}q^{\lambda_\Omega}, |b+z_\Omega a|\le e^{\sqrt{\rho_\Omega}}q^{1-\lambda_{\Omega}}\big\}.
		\end{align*}
		It follows from Minkowski's linear forms theorem that
		$\cW(\Omega,\bv)\ne\emptyset$. We choose and fix $$\bw(\Omega,\bv)=\big(a(\Omega,\bv),b(\Omega,\bv),c(\Omega,\bv)\big)\in\cW(\Omega,\bv)$$ such that
		\begin{align}
			&\ \max\big\{|a(\Omega,\bv)|,|b(\Omega,\bv)+z_\Omega a(\Omega,\bv)|\big\} \notag\\
			=&\ \min\left\{\max\big\{|a|,|b+z_\Omega a|\big\}:(a,b,c)\in\cW(\Omega,\bv)\right\}, \label{E:max}
		\end{align}
		and define
		$$H_\Omega(\bv):=q\max\big\{|a(\Omega,\bv)|,|b(\Omega,\bv)+z_\Omega a(\Omega,\bv)|\big\}.$$
		
		\begin{lemma}\label{L:q}
			For any closed ball $\Omega\subset\Sigma$ and $\bv=(p,r,q)\in\Xi$, we have
			\begin{equation}\label{E:qq}
				q\le H_\Omega(\bv)\le q^{1+\lambda_\Omega}.
			\end{equation}
		\end{lemma}
		
		\begin{proof}
			The first inequality is obvious. By Minkowski linear forms theorem, $\cW(\Omega,\bv)$ contains a vector $(a_0,b_0,c_0)$ with $|a_0|\le q^{\lambda_\Omega}$ and $|b_0+z_\Omega a_0|\le q^{1-\lambda_\Omega}$. Thus, it follows from \eqref{E:max} that
			$$\max\big\{|a(\Omega,\bv)|,|b(\Omega,\bv)+z_\Omega a(\Omega,\bv)|\big\}\le\max\big\{|a_0|,|b_0+z_\Omega a_0|\big\}\le\max\{q^{\lambda_\Omega},q^{1-\lambda_\Omega}\}=q^{\lambda_\Omega}.$$
			Hence the second inequality.
		\end{proof}

		Let $\Omega_0\subset\Sigma$ be a closed ball of radius $\rho_0$. Then $\rho_0< 1/4$. Let $\kappa>2$ be such that
		\begin{equation}\label{E:kappa}
			\max\{|x|,|y|,|z|\}\le \kappa-1, \quad \forall \ (x,y,z)\in B_{\Omega_0}.
		\end{equation}
		Let $\beta\in(0,1)$, and $R$ and $\epsilon$ be positive numbers such that
		\begin{equation}\label{E:R}
			R\ge\max\{100\beta^{-1},10^4\kappa^4\} \text{ and } \sqrt{R^{-1}}\log R^4\le 1,
		\end{equation}
		and
		\begin{equation}\label{E:epsilon}
			\epsilon\le R^{-24}\rho_0.
		\end{equation}
		Let $\sB_0=\{\Omega_0\}$.
		For $n\ge1$, let $\sB_n$ be the family of closed balls defined by
		$$\sB_n:=\{\Omega\subset \Omega_0:\beta R^{-n}\rho_0<\rho_\Omega\le R^{-n}\rho_0\}.$$
		For a closed ball $\Omega$, if
		\begin{equation}\label{E:condition0}
			\text{$\Omega\in\sB_n$ for some $n\ge0$,}
		\end{equation}
		we define
		$$\cV_\Omega:=\{\bv\in\Xi: H_n\le H_\Omega(\bv)\le 3H_{n+1}\},$$
		where
		$$H_n:=20\epsilon \kappa \rho_0^{-1}R^n, \quad n\ge0.$$
		Note that since $R>\beta^{-1}$, the families $\sB_n$ are pairwise disjoint. So \eqref{E:condition0} is satisfied for at most one integer $n$, and hence $\cV_\Omega$ is well-defined.
		It follows from \eqref{E:qq} that if $\bv=(p,r,q)\in\cV_\Omega$, then
		\begin{equation}\label{E:q}
			H_n^{\frac{1}{1+\lambda_{\Omega}}}\le q\le 3H_{n+1}.
		\end{equation}
		Whenever \eqref{E:condition0} is satisfied, we also define
		$$\cV_{\Omega,1}:=\{(p,r,q)\in\cV_\Omega: H_n^{\frac{1}{1+\lambda_{\Omega}}}\le q\le H_n^{\frac{1}{1+\lambda_{\Omega}}}R^{20}\}$$
		and
		$$\cV_{\Omega,k}:=\{(p,r,q)\in\cV_\Omega: H_n^{\frac{1}{1+\lambda_{\Omega}}}R^{2k+16}\le q\le H_n^{\frac{1}{1+\lambda_{\Omega}}}R^{2k+18}\}, \quad k\ge2.$$
		
		\begin{lemma}\label{L:part}
			If $\Omega\in\sB_n$, then $\cV_\Omega=\bigcup_{k=1}^{[\frac{n}{2}]}\cV_{\Omega,k}$.
		\end{lemma}
		
		\begin{proof}
			In view of \eqref{E:q}, it suffices to show that if $k>[\frac{n}{2}]$ then $\cV_{\Omega,k}=\emptyset$. Suppose to the contrary that $\cV_{\Omega,k}\ne\emptyset$ for some $k>[\frac{n}{2}]$. Let $(p,r,q)\in\cV_{\Omega,k}$. Then
			$$3H_{n+1}\ge q\ge H_n^{\frac{1}{1+\lambda_{\Omega}}}R^{2k+16}> H_n^{\frac{1}{1+\lambda_{\Omega}}}R^{n+15}.$$
			But
			$$ H_n^{-\frac{1}{1+\lambda_{\Omega}}}R^{-n-15}\cdot 3H_{n+1}=3(20\epsilon\kappa \rho_0^{-1})^{\frac{\lambda_{\Omega}}{1+\lambda_{\Omega}}}R^{-\frac{1}{1+\lambda_{\Omega}}n-14}<1,$$
			a contradiction.
		\end{proof}

		Next, we inductively define a subfamily $\sB'_n$ of $\sB_n$ as follows. Let $\sB'_0=\{\Omega_0\}$. If $n\ge1$ and $\sB'_{n-1}$ has been defined, we let
		\begin{equation}\label{E:B'}
			\sB'_n:=\Big\{\Omega\in\sB_n:\Omega\subset \Omega' \text{ for some } \Omega'\in\sB'_{n-1},  \text{ and } \Omega\cap\bigcup_{\bv\in\cV_\Omega}\Delta(\bv, \epsilon)=\emptyset\Big\}.
		\end{equation}
		
		The following two lemmas concerning $\sB'_n$ are key steps in the proof of Theorem \ref{T:ru}. Their proofs are technical and postponed to the next two sections.
		
		\begin{lemma}\label{L:main1}
			Let $n\ge1$, $\Omega\in\sB'_n$ and $\bv=(p,r,q)\in\Xi$. If $q^{1+\lambda_\Omega}\le 3H_{n+1}$, then $\Delta(\bv, \epsilon)\cap \Omega=\emptyset$.
		\end{lemma}
		
		\begin{lemma}\label{L:main2}
			Let $n\ge1$, $\Omega\in\sB'_n$, and $k\in [1,n]$. There exists an affine plane $L_k(\Omega)\subset\Sigma$ such that for any $\Omega'\in\sB_{n+k}$ with $\Omega'\subset \Omega$ and any $\bv\in\cV_{\Omega',k}$, we have $$\Delta(\bv, \epsilon)\cap \Omega'\subset L_k(\Omega)^{(R^{-(n+k)}\rho_0)}.$$
		\end{lemma}
		
		We now prove Lemma \ref{statement1} assuming the truth of Lemmas \ref{L:main1} and \ref{L:main2}.
		
		\begin{proof}[Proof of Lemma \ref{statement1}]
			By Lemma \ref{L:HPW}, it suffices to prove that $S'$ is HPW on $\Sigma$. Let $\beta\in(0,1)$, $\gamma>0$. We prove that Alice has a strategy to win the $(\beta,\gamma)$-hyperplane potential game on $\Sigma$ with target set $S'$. In the first round of the game, Bob chooses a closed ball $\Omega_0\subset\Sigma$. By \cite[Remark 2.4]{AGK}, we may without loss of generality assume that Bob will play so that $\rho_i:=\rho(\Omega_i)\to0$.  Let $\kappa$ and $R$ be positive numbers satisfying \eqref{E:kappa} and \eqref{E:R}. We also require that
			\begin{equation}\label{E:R2}
				(R^\gamma-1)^{-1}\le (\beta^2/2)^\gamma.
			\end{equation}
			Let $\epsilon, \sB_n, \cV_\Omega, \cV_{\Omega,k}$ and $\sB'_n$ be as above. Then Lemmas \ref{L:part}--\ref{L:main2} hold.
			Let Alice play according to the following strategy: Suppose that $i\ge0$, and Bob has chosen the closed ball $\Omega_i$. If there is $n\ge0$ such that
			\begin{equation}\label{E:condition}
				\text{$\Omega_i\in\sB'_n$, and $i$ is the smallest nonnegative integer with $\Omega_i\in\sB_n$,}
			\end{equation}
			then Alice chooses the family of neighborhoods $\{L_k(\Omega_i)^{(2R^{-(n+k)}\rho_0)}:k\in\NN\}$, where $L_k(\Omega_i)$ are the planes given in Lemma \ref{L:main2}. Otherwise, Alice makes an arbitrary move. Since $\Omega_i\in\sB_n$, we have that $\rho_i>\beta R^{-n}\rho_0$. This, together with \eqref{E:R2}, implies that
			$$\sum_{k=1}^\infty(2R^{-(n+k)}\rho_0)^\gamma=(2R^{-n}\rho_0)^\gamma(R^\gamma-1)^{-1}\le(\beta \rho_i)^\gamma.$$
			So Alice's move is legal. We prove that this guarantees a win for Alice, that is, the unique point $\bx_\infty=(\lambda_\infty, \omega_\infty)$ in $\bigcap_{i=0}^\infty \Omega_i$ lies in $$S'\cup\bigcup_{n\in\cN}\bigcup_{k=1}^n L_k(\Omega_{i_n})^{(2R^{-(n+k)}\rho_0)},$$ where
			$$\cN:=\{n\ge0:\text{ there exists } i=i_n \text{ such that \eqref{E:condition} holds}\}.$$ There are two cases.
			
			\medskip
			
			\textbf{Case (1).} Assume $\cN=\NN\cup\{0\}$. For any $\bv=(p,r,q)\in\Xi$, there is $n$ such that $q^{2}\le 3H_{n+1}$. Then it follows that  $q^{1+\lambda_n}\le  q^2\le 3H_{n+1}.$
			Since $n\in\cN$, we have that $\Omega_{i_n}\in\sB'_n$. Then Lemma \ref{L:main1} implies that $\Delta(\bv,\epsilon)\cap \Omega_{i_n}=\emptyset$. Thus $\bx_\infty\notin\Delta(\bv, \epsilon)$. It then follows from \eqref{E:S-epsilon} and the arbitrariness of $\bv$ that $\bx_\infty\in S(\epsilon)\subset S'$. Hence Alice wins.
			
			\medskip
			
			\textbf{Case (2).} Assume $\cN\ne\NN\cup\{0\}$. Let $n$ be the smallest nonnegative integer with $n\notin\cN$. Since $B_0\in\sB'_0$, we have $0\in\cN$. Hence $n\ge1$. Since $\rho_i\to0$ and $\rho_{i+1}\ge\beta \rho_i$, there must exist $i\ge1$ with $\Omega_i\in\sB_n$ and $\Omega_{i-1}\notin\sB_n$. It follows that $\Omega_i\notin\sB'_n$. Since $n-1\in\cN$, we have $\Omega_{i_{n-1}}\in\sB'_{n-1}$. Thus it follows from \eqref{E:B'} that  $\Omega_i\cap\bigcup_{\bv\in\cV_{\Omega_i}}\Delta(\bv,\epsilon)\ne\emptyset$. By Lemma \ref{L:part}, there is $k\in [1,[n/2]]$ and $\bv\in\cV_{\Omega_i,k}$ such that $\Omega_i\cap\Delta(\bv, \epsilon)\ne\emptyset$. Applying Lemma \ref{L:main2} to $\Omega=\Omega_{i_{n-k}}$ and $\Omega'=\Omega_i$, we obtain
			$\Omega_i\cap\Delta(\bv,\epsilon)\subset L_k(\Omega_{i_{n-k}})^{(R^{-n}\rho_0)}$.
			So $\Omega_i\cap L_k(\Omega_{i_{n-k}})^{(R^{-n}\rho_0)}\ne\emptyset$. In view of $\rho_i\le R^{-n}\rho_0$, it follows that
			$\bx_\infty\in \Omega_i\subset L_k(\Omega_{i_{n-k}})^{(2R^{-n}\rho_0)}$. Hence Alice also wins.
			
			\medskip
			
			This completes the proof of Theorem \ref{T:ru} modulo Lemmas \ref{L:main1} and \ref{L:main2}.
		\end{proof}

		\section{Proof of Lemma \ref{L:main1}}
		We shall need the following simple but useful lemma.
		\begin{lemma}\label{L:Rvalue}
			Let $n\ge 1$ and $\rho\in (0,R^{-n}\rho_0]$, then
			\begin{equation}\label{Rvalue}
				\rho\log H_{2n+2}\le \frac{1}{2}\sqrt{\rho}\le 1
			\end{equation}
		\end{lemma}
		
		\begin{proof}
			The second inequality is direct. Since the function $x^{-1}\log x$ is monotonically decreasing when $x>e$, it follows that
			$$2\sqrt{\rho}\log H_{2n+2}\le \sqrt{R^{-n}}\log R^{2n+2}\le \sqrt{R^{-n}}\log R^{4n}\le \sqrt{R^{-1}}\log R^4\le 1.$$
			This completes the proof of the first inequality.
		\end{proof}
		\begin{proof}[Proof of Lemma \ref{L:main1}]
			We denote $\Omega_n=\Omega$, and let $\Omega_n\subset\cdots\subset \Omega_1\subset \Omega_0$ be such that $\Omega_k\in\sB'_k$. Write $\lambda_{\Omega_k}=\lambda_k$ and $\rho_{\Omega_k}=\rho_k$ for simplicity. Suppose to the contrary that the conclusion of the lemma is not true. Then there exists $\bv=(p,r,q)\in\ZZ^2\times\NN$ with $q^{1+\lambda_n}\le 3H_{n+1}$ such that $\Delta(\bv,\epsilon)\cap \Omega_k\ne\emptyset$ for every $1\le k\le n$. It follows from the definition of $\sB'_k$ that $\bv\notin\cV_{\Omega_k}$, that is,
			\begin{equation}\label{E:notin}
				H_{\Omega_k}(\bv)\notin[H_k,3H_{k+1}], \qquad 1\le k\le n.
			\end{equation}
			Let $1\le m\le n$ be such that
			\begin{equation}\label{E:n_0}
				H_{m}\le q^{1+\lambda_n}\le H_{m+1}.
			\end{equation}
			We inductively prove that
			\begin{equation}\label{E:induc}
				H_{\Omega_k}(\bv)<H_k, \qquad [\frac{m}{2}]\le k\le m.
			\end{equation}
			By letting $k=[\frac{m}{2}]$ in \eqref{E:induc},  we have
			$$q\le H_{\Omega_{[\frac{m}{2}]}}(\bv)\le 3H_{[\frac{m}{2}]+1}\le \sqrt{H_m}<q,$$
			which is a contradiction. So Lemma \ref{L:main1} follows from \eqref{E:induc}.
			Since
			\begin{align*}
				H_{\Omega_{m}}(\bv) \le q^{1+\lambda_{m}}\le q^{|\lambda_{m}-\lambda_{n}|}q^{1+\lambda_{n}} \le H_{m+1}^{\rho_m} H_{m+1}=e^{\rho_m\log H_{m+1}}H_{m+1}\le 3H_{m+1},
			\end{align*}
			it follows from \eqref{E:notin} that
			\eqref{E:induc} holds for $k=m$. Suppose that $[\frac{m}{2}]\le k\le m-1$ and \eqref{E:induc} holds if $k$ is replaced by $k+1$.
			We prove that
			\begin{equation}\label{E:conti}
				H_{\Omega_k}(\bv)\le 3H_{\Omega_{k+1}}(\bv).
			\end{equation}
			To prove \eqref{E:conti}, the first step is to establish
			\begin{equation}\label{E:k+1}
				\bw(\Omega_{k+1},\bv)\in\cW(\Omega_k,\bv).
			\end{equation}
			Write 
            $\bw(\Omega_k,\bv)=(a_k,b_k,c_k)$ and $z_{\Omega_k}=z_k$.
			Since $\bw(\Omega_{k+1},\bv)\in\cW(\Omega_{k+1},\bv)$, we have 
			$$|a_{k+1}| \le e^{\sqrt{\rho_{k+1}}}q^{\lambda_{k+1}} \quad\text{ and } \quad |b_{k+1}+z_{k+1}a_{k+1}| \le e^{\sqrt{\rho_{k+1}}}q^{1-\lambda_{k+1}}.$$
			Note that $\sqrt{\rho_{k+1}}\le \sqrt{\beta^{-1} R^{-1}}\sqrt{\rho_k}\le \frac{1}{10}\sqrt{\rho_k}$, and by Lemma \ref{L:Rvalue}
            $$
                q^{|\lambda_{k+1}-\lambda_k|}\leq e^{\rho_k\log q} \leq e^{\rho_k\log H_{m+1}}\leq e^{\frac{1}{2}\sqrt{\rho_k}}.
            $$
			Now we are in position to prove \eqref{E:k+1}. Firstly, we have
			\begin{align*}
				|a_{k+1}| &\le e^{\sqrt{\rho_{k+1}}}q^{\lambda_{k+1}}\\
				&\leq\left(e^{\sqrt{\rho_{k+1}}-\sqrt{\rho_{k}}}q^{|\lambda_{k+1}-\lambda_{k}|}\right)e^{\sqrt{\rho_{k}}}q^{\lambda_{k}}\\
				&\le \left(e^{\sqrt{\rho_{k+1}}-\frac{1}{2}\sqrt{\rho_{k}}}\right) e^{\sqrt{\rho_{k}}}q^{\lambda_{k}}
				\\
				&\le e^{\sqrt{\rho_{k}}}q^{\lambda_{k}}.
			\end{align*}
			Secondly, we have
			\begin{align*}
				|b_{k+1}+z_ka_{k+1}|
				&\le |b_{k+1}+z_{k+1}a_{k+1}|+|a_{k+1}||z_k-z_{k+1}|\\
				&\le |b_{k+1}+z_{k+1}a_{k+1}|+|a_{k+1}|\rho_k \\
				&\le \big(1+|a_{k+1}|\rho_k\big)e^{\sqrt{\rho_{k+1}}}q^{1-\lambda_{k+1}} \\
				&\le \left(1+q^{-1}H_{\Omega_{k+1}}(\bv)\rho_k\right) \left(e^{\sqrt{\rho_{k+1}}-\sqrt{\rho_{k}}}q^{|\lambda_{k+1}-\lambda_{k}|}\right)
				e^{\sqrt{\rho_{k}}}q^{1-\lambda_{k}} \\
				&\le \left(1+\sqrt{H_m^{-1}}H_{k+1}\rho_k\right) \left(e^{\sqrt{\rho_{k+1}}-\frac{1}{2}\sqrt{\rho_{k}}}\right)
				e^{\sqrt{\rho_{k}}}q^{1-\lambda_{k}} \\
				&\le \left(1+\sqrt{20\epsilon \kappa \beta^{-1}R^2}\sqrt{\rho_{k+1}}\right) \left(e^{\sqrt{\rho_{k+1}}-\frac{1}{2}\sqrt{\rho_{k}}}\right)
				e^{\sqrt{\rho_{k}}}q^{1-\lambda_{k}} \\
				&\le \left(1+\sqrt{\rho_{k+1}}\right)\left(e^{\sqrt{\rho_{k+1}}-\frac{1}{2}\sqrt{\rho_{k}}}\right) e^{\sqrt{\rho_{k}}}q^{1-\lambda_{k}} \\
				&\le  e^{\sqrt{\rho_{k}}}q^{1-\lambda_{k}}. 
			\end{align*}			
			Hence \eqref{E:k+1} holds. It then follows from \eqref{E:k+1} and the definition of $\bw(\Omega_k,\bv)$ that
			\begin{align*}
				H_{\Omega_k}(\bv)&=q\max\{|a_k|,|b_k+z_ka_k|\}\\
				&\le q\max\{|a_{k+1}|,|b_{k+1}+z_ka_{k+1}|\}\\
				&\le 3q\max\{|a_{k+1}|,|b_{k+1}+z_{k+1}a_{k+1}|\}\\
				&=3H_{\Omega_{k+1}}(\bv).
			\end{align*}
			This proves \eqref{E:conti}.
			
			It follows from \eqref{E:conti} and the induction hypothesis that $H_{\Omega_k}(\bv)< 3H_{k+1}$.
			By \eqref{E:notin}, we have $H_{\Omega_k}(\bv)<H_k$.
			Thus \eqref{E:induc} is proved.			
		\end{proof}

		\section{Proof of Lemma \ref{L:main2}}\label{S:pf_main}
		
		In this section, we prove Lemma \ref{L:main2}. We first establish the following lemma.
		\begin{lemma}\label{L:inc}
			Let $n\ge 1$, $\Omega\in \sB_n$ and $\bv=(p,r,q)\in \Xi$ with $q\le H_{2n+2}$. Then
			\begin{equation}\label{e:inc}
				\left( I_{\Omega}\times\Delta_{\lambda_{\Omega}}(\bv,\epsilon/3)\right)\cap \Omega\subset \Delta(\bv,\epsilon)\cap \Omega \subset \left(I_{\Omega}\times\Delta_{\lambda_{\Omega}}(\bv,3\epsilon)\right)\cap \Omega.
			\end{equation}
		\end{lemma}
		
		\begin{proof}
			Let $(\lambda,x,y,z)\in\Delta(\bv,\epsilon)\cap \Omega$. Then
			$$\Big|x-\frac{p}{q}-z\Big(y-\frac{r}{q}\Big)\Big|<\frac{\epsilon}{q^{1+\lambda}}, \qquad  \Big|y-\frac{r}{q}\Big|<\frac{\epsilon}{q^{2-\lambda}}.$$
			Hence to show \eqref{e:inc}, it suffices to show that, for any $\lambda\in I_{\Omega}$, $q^{|\lambda-\lambda_{\Omega}|}\le 3$. According to Lemma \ref{L:Rvalue}, we have
			$$q^{|\lambda-\lambda_{\Omega}|}\le e^{\rho_{\Omega}\log q}\le e^{\rho_{\Omega}\log H_{2n+2}}\le 3.$$
			This completes the proof.
		\end{proof}
		
		\begin{lemma}\label{L:inequ}
			Let $n\ge 1$, $\Omega\in\sB'_n$, $k\in [1,n]$, and $\Omega_j\in\sB_{n+k}$ with $\Omega_j\subset \Omega$, $j=1,2$.
			Let $\bv_j=(p_j,r_j,q_j)\in\cV_{\Omega_j,k}$ be such that $\Delta(\bv_j,\epsilon)\cap \Omega\ne\emptyset$, and $\bw_j=\bw(\Omega_j,\bv_j)$. Then
			\begin{equation}\label{E:v1w2}
				\max\left\{|\bv_1\cdot\bw_2|, \  |\bv_2\cdot\bw_1|\right\} \le \begin{cases}
					R^{k+22}\epsilon, & k=1,\\
					R^{k+4}\epsilon, & k\ge2.
				\end{cases}
			\end{equation}
		\end{lemma}
		
		\begin{proof}
			For simplicity, we write $\rho_{\Omega}=\rho$ and $\lambda_{\Omega}=\lambda$.   We first give a bound on $q_1$ and $q_2$ in terms of $\lambda$. Indeed, for any $\lambda'\in I_{\Omega}$, we have
			\begin{equation}\label{E:bound-lambda}
				H_{n+k}^{\big|\frac{1}{1+\lambda'}-\frac{1}{1+\lambda}\big|}\le H_{2n}^{|\lambda'-\lambda|} \le e^{\rho\log H_{2n}}\le 3.
			\end{equation}
			By applying \eqref{E:bound-lambda} to $\lambda_{\Omega_1}$ and $\lambda_{\Omega_2}$, we get
			\begin{equation}\label{E:bound-q-1-2}
				\begin{aligned}
					q_1, q_2\in &  \left[\frac{1}{3}H_{n+k}^{\frac{1}{1+\lambda}},\  3H_{n+k}^{\frac{1}{1+\lambda}}R^{20}\right], & k=1, \\
					q_1, q_2\in & \left[\frac{1}{3}H_{n+k}^{\frac{1}{1+\lambda}}R^{2k+16}, \ 3H_{n+k}^{\frac{1}{1+\lambda}}R^{2k+18}\right], & k\ge 2.
				\end{aligned}
			\end{equation}
			Let $(\lambda_j',x_j,y_j,z_j)\in\Delta(\bv_j,\epsilon)\cap \Omega$.
			Since $\Omega_j\in\sB_{n+k}$ and $\bv_j\in \cV_{\Omega_j}$, we have
			$$q_j\le H_{\Omega_j}(\bv_j)\le 3H_{n+k+1} \le 3H_{2n+1}\le H_{2n+2}.$$
			Thus, it follows from Lemma \ref{L:inc} that, $(x_j,y_j,z_j)\in\Delta_{\lambda}(\bv_j,3\epsilon)$. Hence
			$$\Big|x_j-\frac{p_j}{q_j}-z_j\Big(y_j-\frac{r_j}{q_j}\Big)\Big|<\frac{3\epsilon}{q_j^{1+\lambda}}, \qquad  \Big|y_j-\frac{r_j}{q_j}\Big|<\frac{3\epsilon}{q_j^{2-\lambda}}.$$
			The latter inequality implies that
			$$\Big|\frac{r_j}{q_j}\Big|\le|y_j|+\frac{3\epsilon}{q_j^{2-\lambda}}\le \kappa.$$
			Thus
			\begin{align*}
				&\ \Big|\frac{p_1}{q_1}-\frac{p_2}{q_2}-z_{\Omega_2}\Big(\frac{r_1}{q_1}-\frac{r_2}{q_2}\Big)\Big|\\
				=&\ \Big|-\Big(x_1-\frac{p_1}{q_1}-z_1\Big(y_1-\frac{r_1}{q_1}\Big)\Big)+\Big(x_2-\frac{p_2}{q_2}-z_2\Big(y_2-\frac{r_2}{q_2}\Big)\Big)\\
				& \qquad +(x_1-x_2)+\frac{r_1}{q_1}(z_1-z_{\Omega_2})+\frac{r_2}{q_2}(z_{\Omega_2}-z_2)-(y_1z_1-y_2z_2)\Big|\\
				\le&\ \Big|x_1-\frac{p_1}{q_1}-z_1\Big(y_1-\frac{r_1}{q_1}\Big)\Big|+\Big|x_2-\frac{p_2}{q_2}-z_2\Big(y_2-\frac{r_2}{q_2}\Big)\Big|\\
				& \qquad +|x_1-x_2|+\Big|\frac{r_1}{q_1}\Big||z_1-z_{\Omega_2}|+\Big|\frac{r_2}{q_2}\Big||z_{\Omega_2}-z_2|+|y_1||z_1-z_2|+|z_2||y_1-y_2|\\
				\le&\ \frac{3\epsilon}{q_1^{1+\lambda}}+\frac{3\epsilon}{q_2^{1+\lambda}}+10\kappa \rho,
			\end{align*}
			and
			$$
			\Big|\frac{r_1}{q_1}-\frac{r_2}{q_2}\Big|=\Big|-\Big(y_1-\frac{r_1}{q_1}\Big)+\Big(y_2-\frac{r_2}{q_2}\Big)+(y_1-y_2)\Big|
			\le\frac{3\epsilon}{q_1^{2-\lambda}}+\frac{3\epsilon}{q_2^{2-\lambda}}+2\rho.
			$$
			
			Write $\bw_j=(a_j,b_j,c_j)$. Then according to Lemma \ref{L:Rvalue}, we have
			$$|a_j|\le e^{\sqrt{\rho_{\Omega_j}}}q_j^{\lambda_{\Omega_j}}\le 9q_j^\lambda, \quad \text{and} \quad |b_j+z_{\Omega_j}a_j|\le e^{\sqrt{\rho_{\Omega_j}}}q_j^{1-\lambda_{\Omega_j}}\le 9q_j^{1-\lambda}.$$
			Set $d_k$ to be a function such that $d_k=20$ if $k=1$, and $d_k=2$ otherwise.
			It follows that
			\begin{align*}
				|\bv_1\cdot\bw_2|=&\ q_1|(q_1^{-1}\bv_1-q_2^{-1}\bv_2)\cdot\bw_2|\\
				=&\ q_1\Big|a_2\Big(\frac{p_1}{q_1}-\frac{p_2}{q_2}\Big)+b_2\Big(\frac{r_1}{q_1}-\frac{r_2}{q_2}\Big)\Big| \\
				=&\ q_1\Big|a_2\Big(\frac{p_1}{q_1}-\frac{p_2}{q_2}-z_{\Omega_2}\Big(\frac{r_1}{q_1}-\frac{r_2}{q_2}\Big)\Big)
				+(b_2+z_{\Omega_2}a_2)\Big(\frac{r_1}{q_1}-\frac{r_2}{q_2}\Big)\Big| \\
				\le&\ q_1|a_2|\Bigg(\frac{3\epsilon}{q_1^{1+\lambda}}+\frac{3\epsilon}{q_2^{1+\lambda}}+10\kappa \rho\Bigg)+q_1|b_2+z_{\Omega_2}a_2|\Bigg(\frac{3\epsilon}{q_1^{2-\lambda}}+\frac{3\epsilon}{q_2^{2-\lambda}}+2\rho\Bigg) \\
				\le&\ 9q_1q_2^{\lambda}\Bigg(\frac{3\epsilon}{q_1^{1+\lambda}}+\frac{3\epsilon}{q_2^{1+\lambda}}\Bigg)
				+9q_1q_2^{1-\lambda}\Bigg(\frac{3\epsilon}{q_1^{2-\lambda}}+\frac{3\epsilon}{q_2^{2-\lambda}}\Bigg) \\
				&\ +12\kappa \rho q_1\max\big\{|a_2|,|b_2+z_{\Omega_2}a_2|\big\}\\
				\le&\  27\epsilon \Bigg(\frac{q_2^{\lambda}}{q_1^{\lambda}}+\frac{q_1}{q_2}+\frac{q_2^{1-\lambda}}{q_1^{1-\lambda}}
				+\frac{q_1}{q_2}\Bigg)+12\kappa R^{-n}\rho_0H_{\Omega_2}(\bv_2)\frac{q_1}{q_2} \\
				\le &\ 108R^{d_k}\epsilon+720\kappa^2R^{d_k+k+1}\epsilon \le R^{d_k+k+2}\epsilon.
			\end{align*}
			This proves \eqref{E:v1w2} for $|\bv_1\cdot\bw_2|$. A similar argument shows the same inequality for $|\bv_2\cdot\bw_1|$.
		\end{proof}
		
		For a closed ball $\Omega\subset\Sigma$ and $\bv=(p,r,q)\in\Xi$, we consider the plane
		\begin{equation}\label{E:plane}
        L(\Omega,\bv)=\{(\lambda,x,y,z)\in\Sigma:a(\Omega,\bv)x+b(\Omega,\bv)y+c(\Omega,\bv)=0\}.
        \end{equation}
		The next lemma states that many pairs $(\Omega,\bv)$ share the same hyperplane $L(\Omega,\bv)$.
		
		\begin{lemma}\label{L:const}
			Let $n\ge1$, $\Omega\in\sB'_n$, $k\in [1,n]$. Then either 
            \begin{itemize}
                \item[(i)] there is a plane $L_k(\Omega)$ such that for any $\Omega'\in\sB_{n+k}$ with $\Omega'\subset \Omega$, if $\bv\in\cV_{\Omega',k}$ and $\Delta(\bv,\epsilon)\cap \Omega\ne\emptyset$, then $L(\Omega',\bv)=L_k(\Omega)$, or
                \item[(ii)] $k=1$, and there is $\bv_0=(p_0,r_0,q_0)\in \Xi$ with $\frac{1}{3}H_{n+1}^{\frac{1}{1+\lambda_{\Omega}}}\leq q_0\leq 3H_{n+2}$ such that for any $\Omega'\in \sB_{n+1}$ with $\Omega'\subset \Omega$, if $\bv\in\cV_{\Omega',1}$ and $\Delta(\bv,\epsilon)\cap \Omega\neq \varnothing$, then $\bv=\bv_0$.
            \end{itemize}
		\end{lemma}

		\begin{proof}
			
            We proceed by considering two separate cases.
			
			\medskip
			
			\textbf{Case (1).} Suppose $k=1$. We consider two subcases.

            \textbf{Case (1.1).} The linear span of the set $$
            \cV:=\bigcup_{\substack{\Omega'\in \sB_{n+1}\\ \Omega'\subset\Omega}}\{\bv\in \cV_{\Omega',1}:\Delta(\bv,\epsilon)\cap \Omega\neq \varnothing\}
            $$
            is of dimension $\geq 2$. We prove that (i) holds. It suffices to prove that for $j=1,2$, if $\Omega_j\in\sB_{n+1}$, $\Omega_j\subset \Omega$, $\bv_j\in\cV_{\Omega_j,1}$, $\Delta(\bv_j, \epsilon)\cap \Omega\ne\emptyset$, and $\bv_1$ and $\bv_2$ are linearly independent, then $L(\Omega_1,\bv_1)=L(\Omega_2,\bv_2)$. Write $\bw_j=\bw(\Omega_j,\bv_j)$.
            It follows from Lemma \ref{L:inequ} that
			$$\max\{|\bv_1\cdot\bw_2|, |\bv_2\cdot\bw_1|\}\le R^{23}\epsilon<1.$$
			Since both $\bv_2\cdot\bw_1$ and $\bv_1\cdot\bw_2$ are integers, they must be $0$. 
            Hence both $\bw_1$ and $\bw_2$ are orthogonal to $\mathrm{span}\{\bv_1,\bv_2\}$.
            Thus $\bw_1$ and $\bw_2$ are linearly dependent. This means that $L(\Omega_1,\bv_1)=L(\Omega_2,\bv_2)$.

            \textbf{Case (1.2).} The linear span of $\cV$ is of dimension $1$. We prove that (ii) holds. Let $\bv_0=(p_0,r_0,q_0)$ be any element in $\cV$. It follows from $\cV\subset \Xi$ that $\cV=\{\bv_0\}$. Moreover, since $\bv_0\in \cV_{\Omega',1}$ for some $\Omega'\in\sB_{n+1}$ with $\Omega'\subset \Omega$, we have $\frac{1}{3}H_{n+1}^{\frac{1}{1+\lambda_{\Omega}}}\leq q_0\leq 3H_{n+2}$. This completes the proof of Case (1).

			\medskip
			
			\textbf{Case (2).} Suppose $k\ge2$. 
            We prove that (i) holds. It suffices to prove that if $\Omega_j\in\sB_{n+k}$, $\Omega_j\subset \Omega$, $\bv_j\in\cV_{\Omega_j,k}$, $\Delta(\bv_j, \epsilon)\cap \Omega\ne\emptyset$, and $\bw_j=\bw(\Omega_j,\bv_j)$  ($j=1,2$), then $\bw_1$ and $\bw_2$ are linearly dependent. 
            Let $\bv_j=(p_j,r_j,q_j)$, $\bw_j=(a_j,b_j,c_j)$. Then
			\begin{align*}
				\max\{|a_j|,|b_j+z_{\Omega_j}a_j|\}&=q_j^{-1}H_{\Omega_j}(\bv_j)\le\left(\frac{1}{3}H_{n+k}^{\frac{1}{1+\lambda}}R^{2k+16}\right)^{-1}\cdot 3H_{n+k+1}\\
				&=9H_{n+k}^{\frac{\lambda}{1+\lambda}}R^{-2k-15}\le H_{n+k}^{\frac{\lambda}{1+\lambda}}R^{-2k-14}. \\
			\end{align*}
			Moreover, for any $(x,y,z)\in B_\Omega$, we have
			\begin{align*}
				|b_j+za_j|\le&\ |b_j+z_{\Omega_j}a_j|+|a_j||z-z_{\Omega_j}|
				\le 3q_j^{1-\lambda_j}+2\rho|a_j|\\
				\le&\ 9q_j^{1-\lambda}+H_{n+k}^{\frac{\lambda}{1+\lambda}}R^{-2k-14}\cdot 2R^{-n}\rho_0 \le 9q_j^{1-\lambda}+2H_{n+k}R^{-n-2k-14}\rho_0 \\
				\le&\  9q_j^{1-\lambda}+1\le 10q_j^{1-\lambda}\le 10\left(3H_{n+k}^{\frac{1}{1+\lambda}}R^{2k+18}\right)^{1-\lambda}\\
				\le&\  30H_{n+k}^{\frac{1-\lambda}{1+\lambda}}R^{k+9} \le H_{n+k}^{\frac{1-\lambda}{1+\lambda}}R^{k+10}.
			\end{align*}
			
			Let $$\bv_0=(p_0,r_0,q_0)=\bw_1\times \bw_2.$$
			It suffices to prove that $\bv_0=0$. By the triple cross product expansion, we have
			$$\bv_1\times \bv_0=\bv_1\times(\bw_1\times \bw_2)=(\bv_1\cdot\bw_2)\bw_1.$$
			Comparing the first two components of the vectors on both sides, we obtain
			\begin{align}
				q_0\frac{r_1}{q_1}-r_0&=q_1^{-1}(\bv_1\cdot\bw_2)a_1, \label{E:c1}\\
				q_0\frac{p_1}{q_1}-p_0&=-q_1^{-1}(\bv_1\cdot\bw_2)b_1. \label{E:c2}
			\end{align}
			Suppose to the contrary that $\bv_0\ne0$. There are two cases.
			
			\medskip
			
			\textbf{Case (2.1).} Suppose $q_0=0$. Then $(p_0,r_0)\ne(0,0)$. It then follows from \eqref{E:c1} and \eqref{E:c2} that
			\begin{align*}
				1&\le\max\{|r_0|,|p_0-z_{\Omega_1}r_0|\}\\
				&=q_1^{-1}|\bv_1\cdot\bw_2|\max\{|a_1|,|b_1+z_{\Omega_1}a_1|\}\\
				&\le \left(\frac{1}{3}H_{n+k}^{\frac{1}{1+\lambda}}R^{2k+16}\right)^{-1}\cdot  R^{k+4}\epsilon \cdot H_{n+k}^{\frac{\lambda}{1+\lambda}}R^{-2k-14}\\
				&\le 3 H_{n+k}^{\frac{\lambda-1}{1+\lambda}}R^{-3k-26}\epsilon<1,
			\end{align*}
			a contradiction.
			
			\medskip
			
			\textbf{Case (2.2).} Suppose $q_0\ne0$. Without loss of generality, we may assume that $q_0>0$. Then
			\begin{align}
				q_0&=|a_1b_2-a_2b_1|=|a_1(b_2+za_2)-a_2(b_1+za_1)|\notag\\
				&\le 2\cdot H_{n+k}^{\frac{\lambda}{1+\lambda}}R^{-2k-14}\cdot H_{n+k}^{\frac{1-\lambda}{1+\lambda}}R^{k+10}=
				2H_{n+k}^{\frac{1}{1+\lambda}}R^{-k-4}.\label{E:q_0}
			\end{align}
			It follows that
			$$q_0/q_1\le2H_{n+k}^{\frac{1}{1+\lambda}}R^{-k-4}\cdot 3H_{n+k}^{-\frac{1}{1+\lambda}}R^{-2k-16}=6R^{-3k-20}\le 1/6.$$
			We prove that $\Delta(\bv_1,\epsilon)\cap \Omega\subset\Delta(\bv_0,\epsilon)$. In view of \eqref{E:q_0} and Lemma \ref{L:inc}, it suffices to show that
			\begin{equation*}
				\Delta_{\lambda}(\bv_1,3\epsilon)\cap B_\Omega\subset \Delta_{\lambda}(\bv_0,\epsilon/3).
			\end{equation*}
			Let $( x,y,z)\in\Delta_{\lambda}(\bv_1,3\epsilon)\cap B_\Omega$.
			It follows from \eqref{E:c1} and \eqref{E:c2}  that
			\begin{align*}
				&\ q_0^{1+\lambda}\Big|x-\frac{p_0}{q_0}-z\Big(y-\frac{r_0}{q_0}\Big)\Big|\\
				\le &\ q_0^{1+\lambda}\Big|x-\frac{p_1}{q_1}-z\Big(y-\frac{r_1}{q_1}\Big)\Big|
				+q_0^{1+\lambda}\Big|\frac{p_1}{q_1}-\frac{p_0}{q_0}-z\Big(\frac{r_1}{q_1}-\frac{r_0}{q_0}\Big)\Big|\\
				< &\ q_0^{1+\lambda}\frac{\epsilon}{q_1^{1+\lambda}}+q_0^{\lambda} q_1^{-1}|b_1+za_1||\bv_1\cdot\bw_2|\\
				\le &\ \epsilon/6+2H_{n+k}^{\frac{\lambda}{1+\lambda}}\cdot 3H_{n+k}^{-\frac{1}{1+\lambda}}R^{-2k-16}\cdot H_{n+k}^{\frac{1-\lambda}{1+\lambda}}R^{k+10}\cdot R^{k+4}\epsilon\\
				\le &\ (1/6+6R^{-2})\epsilon<\epsilon/3,
			\end{align*}
			and
			\begin{align*}
				q_0^{2-\lambda}\Big|y-\frac{r_0}{q_0}\Big|
				\le &\ q_0^{2-\lambda}\Big|y-\frac{r_1}{q_1}\Big|+q_0^{2-\lambda}\Big|\frac{r_1}{q_1}-\frac{r_0}{q_0}\Big|\\
				< &\ q_0^{2-\lambda}\frac{\epsilon}{q_1^{2-\lambda}}+q_0^{1-\lambda} q_1^{-1}|a_1||\bv_1\cdot\bw_2|\\
				\le &\ \epsilon/6+2H_{n+k}^{\frac{1-\lambda}{1+\lambda}}\cdot 3H_{n+k}^{-\frac{1}{1+\lambda}}R^{-2k-16}\cdot H_{n+k}^{\frac{\lambda}{1+\lambda}}R^{-2k-14}\cdot R^{k+4}\epsilon\\
				\le &\ (1/6+ 6R^{-26})\epsilon<\epsilon/3.
			\end{align*}
			Thus $(x,y,z)\in\Delta_\lambda(\bv_0,\epsilon/3)$. This proves $\Delta(\bv_1,\epsilon)\cap \Omega\subset\Delta(\bv_0,\epsilon)$.
			It then follows from $\Delta(\bv_1,\epsilon)\cap \Omega\ne\emptyset$ that $\Delta(\bv_0,\epsilon)\cap \Omega\ne\emptyset$. By Lemma \ref{L:main1}, we have
			$q_0^{1+\lambda}>3H_{n+1}$. But in view of \eqref{E:q_0}, we have
			$$q_0^{1+\lambda}\le 2^{1+\lambda}H_{n+k}R^{-(1+\lambda)(k+4)}\le 4R^{-4}H_{n+1}\le 3H_{n+1},$$
			which leads to a contradiction. This completes the proof of the lemma.
		\end{proof}
		
		We are now prepared to prove Lemma \ref{L:main2}.
		
		\begin{proof}[Proof of Lemma \ref{L:main2}]
			Let $n\ge1$, $\Omega\in\sB'_n$, $k\ge 1$. We need to prove that there is a plane $L_k(\Omega)$ such that for any $\Omega'\in\sB_{n+k}$ with $\Omega'\subset \Omega$ and any $\bv\in\cV_{\Omega',k}$, we have $\Delta(\bv,\epsilon)\cap \Omega'\subset L_k(\Omega)^{(R^{-(n+k)}\rho_0)}$. We need only to consider the case that $\Delta(\bv,\epsilon)\cap \Omega'\ne\emptyset$.  
            By Lemma \ref{L:const}, one of the following statements holds:
            \begin{itemize}
                \item[(i)] there is a plane $L_k(\Omega)$ such that for any $\Omega'\in\sB_{n+k}$ with $\Omega'\subset \Omega$, if $\bv\in\cV_{\Omega',k}$ and $\Delta(\bv,\epsilon)\cap \Omega\ne\emptyset$, then $L(\Omega',\bv)=L_k(\Omega)$.
                \item[(ii)] $k=1$, and there is $\bv_0=(p_0,r_0,q_0)\in \Xi$ with $\frac{1}{3}H_{n+1}^{\frac{1}{1+\lambda_{\Omega}}}\leq q_0\leq 3H_{n+2}$ such that for any $\Omega'\in \sB_{n+1}$ with $\Omega'\subset \Omega$, if $\bv\in\cV_{\Omega',1}$ and $\Delta(\bv,\epsilon)\cap \Omega\neq \varnothing$, then $\bv=\bv_0$.
            \end{itemize}
            
            Suppose that (i) holds.
            Let us fix such a $\Omega'$ and denote $\lambda_{\Omega'}=\lambda',\;\rho_{\Omega'}=\rho'$.
            By Lemma \ref{L:inc}, it suffices to show that
			\begin{equation}\label{E:final}
				\text{$\left(I_{\Omega'}\times\Delta_{\lambda'}(\bv, 3\epsilon)\right)\cap \Omega'\subset L(\Omega',\bv)^{(R^{-(n+k)}\rho_0)}$ for any $\bv=(p,r,q)\in\cV_{\Omega'}$.}
			\end{equation}
			
			Firstly, we have 
			$$\rho q\le R^{-(n+k)}\rho_0\cdot3H_{n+k+1}=60\epsilon \kappa R\le 1/4.$$
			Denote $\bw(\Omega',\bv)=(a,b,c)$. If $(x,y,z)\in\Delta_{\lambda'}(\bv, 3\epsilon)\cap B_{\Omega'}$, then
			\begin{align*}
            |ax+by+c|&=\Big|a\Big(x-\frac{p}{q}-z\Big(y-\frac{r}{q}\Big)\Big)+(b+za)\Big(y-\frac{r}{q}\Big)\Big|\\
				&< |a|\frac{3\epsilon}{q^{1+\lambda'}}+|b+za|\frac{3\epsilon}{q^{2-\lambda'}}\\
				&\le 3\epsilon q^{-1}(|a|q^{-\lambda'}+|b+z_{\Omega'} a|q^{-1+\lambda'}+|z-z_{\Omega'}||a|q^{-1+\lambda'})\\
				&\le 3\epsilon q^{-1}\Big(3+3+\rho' q^{-1+2\lambda'}\Big)\\
				&\le 3\epsilon q^{-1}\Big(6+\rho' q\Big)\\
				&\le 20\epsilon q^{-1}.
			\end{align*}
			Note that
			$$H_{\Omega'}(\bv)=q\max\{|a|,|b+z_{\Omega'}a|\}\le\kappa q\max\{|a|,|b|\}.$$
			Thus
			$$\frac{|ax+by+c|}{\sqrt{a^2+b^2}}<\frac{20\epsilon}{q\max\{|a|,|b|\}}\le\frac{20\epsilon \kappa }{H_{\Omega'}(\bv)}\le\frac{20\epsilon \kappa }{H_{n+k}}=R^{-(n+k)}\rho_0.$$
			This means that $I_{\Omega'}\times\{(x,y,z)\}\subset L(\Omega',\bv)^{(R^{-(n+k)}\rho_0)}$. This proves \eqref{E:final}.

            Suppose that (ii) holds. Let us write $\lambda_{\Omega}=\lambda,\;\rho_{\Omega}=\rho$. Consider the plane $$
            L_1(\Omega)=\left\{(\lambda,x,y,z)\in \Sigma: x-\frac{p_0}{q_0}-z_{\Omega}\left(y-\frac{r_0}{q_0}\right)=0\right\}.
            $$
            We proceed by showing that $\Delta(\bv_0,\epsilon)\cap \Omega\subset L_1(\Omega)^{(R^{-n+1}\rho_0)}$. By Lemma \ref{L:inc}, it suffices to show that 
            \begin{equation}\label{E:final2}
                (I_{\Omega}\times \Delta_{\lambda}(\bv_0,3\epsilon))\subset L_1(\Omega)^{(R^{-(n+1)\rho_0})}.
            \end{equation}  
            If $(x,y,z)\in \Delta_{\lambda}(\bv_0,3\epsilon)\cap B_{\Omega}$, then 
\begin{align*}
				(1+z_{\Omega}^2)^{-1/2}\left|x-\frac{p_0}{q_0}-z_{\Omega}\left(y-\frac{r_0}{q_0}\right)\right|
                &\leq\left|x-\frac{p_0}{q_0}-z\left(y-\frac{r_0}{q_0}\right)\right|+|z-z_{\Omega}|\left|y-\frac{r_0}{q_0}\right|\\
				&< \frac{3\epsilon}{q_0^{1+\lambda}}+\rho\frac{3\epsilon}{q_0^{2-\lambda}}\\
				&\le 3\epsilon q_0^{-(1+\lambda)}(1+\rho q_0)\\
				&\le 27\epsilon H_{n+1}^{-1}(1+R^{-n}\rho_0\cdot 3H_{n+2})\\
                &=27(20\kappa)^{-1}R^{-(n+1)}\rho_0(1+60\epsilon\kappa R^2)\\
				&< R^{-(n+1)}\rho_0.
			\end{align*}
            This means that $I_{\Omega}\times \{(x,y,z)\}\subset L_1(\Omega)^{(R^{-(n+1)\rho_0})}.$ This proves (\ref{E:final2}).
		\end{proof}

		\section*{Acknowledgment}
		Both authors would like to thank Jinpeng An and Dmitry Kleinbock for their helpful discussions and considerable suggestions. L. Guan is supported by National Key R\&D Program of China No. 2022YFA1007800, NSFC and Zhejiang Province  2022XHSJJ010.


\begin{thebibliography}{99}
			
			\bibitem{An1} J.\ An, \textsl{Badziahin-Pollington-Velani's theorem and Schmidt's game}, Bull. Lond. Math. Soc. {\bf 45} (2013), no. 4, 721--733.
			
			\bibitem{An2} \bysame, \textsl{2-dimensional badly approximable vectors and Schmidt's game}, Duke Math. J. {\bf 165} (2016), no. 2, 267--284.
			
			\bibitem{ABV} J.\ An, V.\ Beresnevich, S.\ Velani, \textsl{Badly approximable points on planar curves and winning}, Adv. Math. {\bf 324} (2018), 148--202.

            \bibitem{AGGL} J. An, A. Ghosh, L. Guan, T. Ly, \textsl{Bounded orbits of diagonalizable flows on finite volume quotients of products of $\SL(2,\RR)$}, Adv. Math. {\bf 354} (2019), article no. 106743, 18 pp.
            
			\bibitem{AGK}  J.\ An, L.\ Guan, D.\ Kleinbock, \textsl{Bounded orbits of diagonalizable flows on $\SL_3(\RR)/\SL_3(\ZZ)$}, Int. Math. Res. Not. IMRN {\bf 2015} (2015), no. 24, 13623--13652.

            \bibitem{AGK2} \bysame, \textsl{Nondense orbits on homogeneous spaces and applications to geometry and number theory}, Ergodic Theory Dynam. Systems, {\bf 42} (2022), no. 4, 1327--1372.

            
			\bibitem{Ara1} C. S. Aravinda, \textit{Bounded geodesics and Hausdorff dimension}, in: Mathematical Proceedings of the Cambridge Philosophical Society, vol. {\bf 116}, no. 3, pp. 505-511. Cambridge University Press, 1994.
			
			\bibitem{Ara2} C. S. Aravinda, E. Leuzinger, \textit{Bounded geodesics in rank-1 locally symmetric spaces}, Ergodic Theory and Dynamical Systems {\bf 15}, no. 5 (1995): 813-820.

           \bibitem{BPV} D.\ Badziahin, A.\ Pollington, S.\ Velani, \textsl{On a problem in simultaneous Diophantine approximation: Schmidt's conjecture}, Ann.\ of Math.\ (2) {\bf 174} (2011), no.\ 3, 1837--1883.
            
			\bibitem{Be}  V.\ Beresnevich, \textsl{Badly approximable points on manifolds}, Inventiones mathematicae {\bf 202}, no. 3 (2015): 1199-1240.

            \bibitem{BNY1} V. Beresnevich, E. Nesharim, L. Yang, \textsl{$\Bad(\bw)$ is hyperplane absolute winning}, Geometric and Functional Analysis {\bf 31} (2021): 1-33.

            \bibitem{BNY2} \bysame, \textsl{Winning property of badly approximable points on curves}, Duke Mathematical Journal {\bf 171}, no. 14 (2022): 2841-2880.

            
			\bibitem{BBFKW} R.\ Broderick, Y.\ Bugeaud, L.\ Fishman, D.\ Kleinbock, B.\ Weiss, \textsl{Schmidt's game, fractals, and numbers normal to no base}, Math.\ Res\ Letters, {\bf 17} (2010), no.\ 2, 307--321.
			
			
			
			\bibitem{BFK} R.\ Broderick, L.\ Fishman, D.\ Kleinbock, \textsl{Schmidt's game, fractals, and orbits of toral endomorphisms},  Ergodic Theory Dynam.\ Systems {\bf 31}  (2011),  no.\ 4, 1095--1107.
			
			
			
			\bibitem{BFKRW} R.\ Broderick, L.\ Fishman, D.\ Kleinbock, A.\ Reich, B.\ Weiss, \textsl{The set of badly approximable vectors is strongly $C^1$ incompressible},
			Math.\ Proc.\ Cambridge Philos.\ Soc.\ {\bf 153} (2012), no.\ 2, 319--339.
			
			
			\bibitem{BFS} R.\ Broderick, L.\ Fishman,  D.\ Simmons, \textsl{Badly approximable systems of affine forms and incompressibility on fractals}, J.\  Number Theory {\bf 133}, no. 7 (2013), 2186--2205.
			
			
			\bibitem{Ca} J. W. S.\ Cassels, \textsl{An introduction to the geometry of numbers}, Corrected reprint of the 1971 edition, Springer-Verlag, Berlin, 1997.
			
			\bibitem{CCM11} J. Chaika, Y.  Cheung, and H. Masur, \textit{Winning games for bounded geodesics in Teichmüller discs},  J.\ Mod.\ Dyn.\ {\bf 7} (2013), no. 3, 395--427.

            \bibitem{Da0} S. G.\ Dani, \textsl{Invariant measures of horospherical flows on noncompact homogeneous spaces}, Inventiones mathematicae {\bf 47}, no. 2 (1978): 101-138.
            
			\bibitem{Da1} S. G.\ Dani, \textsl{Divergent trajectories of flows on homogeneous spaces and Diophantine approximation}, J.\ Reine Angew.\ Math.\ {\bf 359} (1985), 55--89.
			
			\bibitem{Da2}  \bysame, \textsl{Bounded orbits of flows on \hs
				s}, Comment.\ Math.\ Helv.\ {\bf 61} (1986), 636--660.
			
			\bibitem{Da3} \bysame,  \textsl{On orbits of endomorphisms of tori and the Schmidt game}, Ergodic Theory Dynam.\ Systems {\bf 8} (1988), 523--529.
			
			
			\bibitem{FSU} L.\ Fishman, D.\ Simmons, M.\ Urba\'{n}ski, \textsl{Diophantine approximation and the geometry of limit sets in Gromov hyperbolic metric spaces},  vol. {\bf 254}, no. 1215. American Mathematical Society, 2018.
			
			

			
			\bibitem{Kl} D.\ Kleinbock, \textsl{Flows on homogeneous spaces and Diophantine properties of matrices}, Duke Math.\ J.\ {\bf 95} (1998), no.\ 1, 107--124.
			
			\bibitem{KM} D.\ Kleinbock, G. A.\ Margulis, \textsl{Bounded orbits of nonquasiunipotent flows on homogeneous spaces}, Amer.\ Math.\ Soc.\ Transl.\ {\bf 171} (1996), 141--172.
			
			
			
			
			\bibitem{KW2} D.\ Kleinbock, B.\ Weiss, \textsl{
				Modified Schmidt games and diophantine approximation
				with weights}, Advances in Math.\  {\bf 223} (2010), 1276--1298.
			
			
			\bibitem{KW04} \bysame, \textit{Bounded geodesics in moduli space}, International Mathematics Research Notices {\bf 2004}, no. 30 (2004): 1551-1560.
			
			\bibitem{KW1} \bysame, 
			\textsl{Modified Schmidt games and  a conjecture of Margulis}, J.\ Mod.\ Dyn.\  {\bf 7}, no.\ 3 (2013), 429--460.
			
			
			
			\bibitem{KW3}\bysame,  \textsl{Values of binary quadratic forms at integer points and Schmidt games}, 
			in: {\bf Recent trends in ergodic theory and dynamical systems (Vadodara, 2012)}, pp.\  77--92, Contemp.\ Math. {\bf 631}, Amer.\ Math.\ Soc., Providence, RI, 2015.
			
			\bibitem {Ma}   G. A.\  Margulis, \textsl{Dynamical and ergodic properties
				of subgroup actions on \hs s with applications to number
				theory},  in: {\bf Proceedings of the International Congress of
				Mathematicians (Kyoto, 1990)}, pp.\ 193--215, Math.\ Soc.\ Japan, Tokyo, 1991.
			
			\bibitem{Mau}  F. I.\ Mautner, \textsl{Geodesic flows on symmetric Riemann spaces}, Annals of Mathematics {\bf 65}, no. 3 (1957): 416-431.
			
			
			\bibitem{Mc}  C. T.\ McMullen, \textsl{Winning sets, quasiconformal maps and Diophantine approximation}, Geom.\ Funct.\ Anal.\ {\bf 20} (2010), no.\ 3, 726--740.
			
			
			
			\bibitem{NS} E.\ Nesharim, D.\ Simmons, \textsl{$\Bad(s,t)$ is hyperplane absolute winning}, Acta Arith.\ {\bf 164} (2014), no.\ 2, 145--152.
			
			
			\bibitem{PV} A.\ Pollington, S.\ Velani, \textsl{On simultaneously badly approximable numbers},
			J.\ London Math.\ Soc. (2) {\bf 66} (2002), no.\ 1, 29--40.
			
			
			
			\bibitem{Sc0}
			W. M.\ Schmidt,  \textsl{A metrical theorem in diophantine approximation},  Canad.\ J.\ Math.\  {\bf  12} (1960), 619--631.
			
			\bibitem{Sc1}\bysame,  \textsl{On badly approximable numbers and certain games}, Trans.\ Amer.\ Math.\ Soc.\ {\bf 123} (1966), 178--199.
			
			\bibitem{Sc2}    \bysame,  \textsl{Badly approximable systems of
				linear forms}, J.\ Number Theory {\bf 1} (1969), 139--154.
			
			
			\bibitem{Sc3} \bysame, \textsl{Open problems in Diophantine approximation}, in: {\bf Diophantine approximations and transcendental numbers (Luminy, 1982)}, pp.\ 271--287, Progr.\ Math.\ 31, Birkh\"{a}user, Boston, 1983.
			
			\bibitem{Sch00} V. Schroeder, \textit{Bounded geodesics in manifolds of negative curvature}, Mathematische Zeitschrift {\bf 235}, no. 4 (2000): 817-828.
			
			%
			%
			
			
		\end{thebibliography}
	\end{document}